\documentclass[11pt,a4paper]{article}
\usepackage{adjustbox}
\usepackage{aligned-overset}
\usepackage{amsmath,amsthm}
\usepackage{authblk}
\usepackage[style=numeric-comp,maxbibnames=99,bibencoding=utf8,firstinits=true]{biblatex}
\usepackage{booktabs}
\usepackage{cancel}
\usepackage[hmargin={26mm,26mm},vmargin={30mm,35mm}]{geometry}
\usepackage{hyperref}
\usepackage{multirow}
\usepackage{rotating}
\usepackage{subcaption}
\usepackage{xcolor}

\usepackage{newtxtext}
\usepackage{newtxmath}

\graphicspath{{./figures/}}

\bibliography{hypre}

\newcommand{\email}[1]{\href{mailto:#1}{#1}}

\theoremstyle{plain}
\newtheorem{theorem}{Theorem}
\newtheorem{proposition}[theorem]{Proposition}
\newtheorem{lemma}[theorem]{Lemma}

\theoremstyle{remark}
\newtheorem{remark}[theorem]{Remark}

\theoremstyle{definition}
\newtheorem{assumption}{Assumption}

\theoremstyle{example}

\DeclareMathOperator{\tr}{tr}
\DeclareMathOperator{\curl}{curl}
\DeclareMathOperator{\rot}{rot}

\begin{document}

\title{Stability, convergence, and pressure-robustness of numerical schemes for incompressible flows with hybrid velocity and pressure}
\author[1]{Lorenzo Botti}
\author[2]{Michele Botti}
\author[3]{Daniele A.~Di Pietro}
\author[1]{Francesco Carlo Massa}
\affil[1]{Dipartimento di Ingegneria e Scienze Applicate, Università degli Studi di Bergamo, viale Marconi 5, 24044 Dalmine, Bergamo, Italy, \email{lorenzo.botti@unibg.it}, \email{francescocarlo.massa@unibg.it}}
\affil[2]{MOX-Dipartimento di Matematica, Politecnico di Milano, Piazza Leonardo da Vinci 32, Milan, Italy, \email{michele.botti@polimi.it}}
\affil[3]{IMAG, Univ. Montpellier, CNRS, Montpellier, France, \email{daniele.di-pietro@umontpellier.fr}}

\maketitle

\begin{abstract}
  In this work we study the stability, convergence, and pressure-robustness of discretization methods for incompressible flows with hybrid velocity and pressure.
  Specifically, focusing on the Stokes problem, we identify a set of assumptions that yield inf-sup stability as well as error estimates which distinguish the velocity- and pressure-related contributions to the error.
  We additionally identify the key properties under which the pressure-related contributions vanish in the estimate of the velocity, thus leading to pressure-robustness.
  Several examples of existing and new schemes that fit into the framework are provided, and extensive numerical validation of the theoretical properties is provided.%
\end{abstract}



\section{Introduction}

In this paper we study the stability, convergence, and pressure-robustness of discretization methods for incompressible flows with hybrid velocity and pressure.
Specifically, focusing on the Stokes problem, we identify a set of assumptions that yield (standard or generalized) inf-sup stability as well as error estimates which distinguish the velocity- and pressure-related contributions to the error.
We additionally identify the key properties under which the pressure-related contributions vanish in the estimate of the velocity, thus leading to pressure-robustness in the sense of \cite{Linke:14}. Several examples of existing and new schemes that fit into the framework are provided.

The use of hybrid approximations of the velocity and discontinuous approximations of the pressure in the context of finite element approximations of incompressible flows dates back to the seminal contribution of Crouzeix and Raviart \cite{Crouzeix.Raviart:73}.
Combined with ideas originating from Discontinuous Galerkin methods, this approach later gave rise to a variety of methods \cite{Brenner:14}, including Hybridizable Discontinuous Galerkin methods \cite{Lehrenfeld.Schoberl:16,Cockburn.Gopalakrishnan:09,Cockburn.Gopalakrishnan.ea:11,Cesmelioglu.Cockburn.ea:17}, Hybrid High-Order (HHO) methods \cite{Aghili.Boyaval.ea:15,Di-Pietro.Ern.ea:16*1,Botti.Di-Pietro.ea:18,Di-Pietro.Droniou:23*2} and nonconforming Virtual Element methods \cite{Ayuso-de-Dios.Lipnikov.ea:16,Zhao.Zhang.ea:20}.
More recently, some authors have pointed out the interest of combining hybrid approximations of both the velocity and pressure \cite{Rhebergen.Wells:18,Kirk.Rhebergen:19,Botti.Massa:22,Baier.Rhebergen.ea:21}, as this can lead to methods that yield $H(\operatorname{div})$-conforming approximations of the velocity and possibly exhibit a better behaviour in the quasi-inviscid limit.

The goal of the present work is precisely to provide a rigorous framework of analysis for such methods focusing on the Stokes problem.
To this purpose, we use as a starting point a fully discrete presentation closely inspired by HHO methods and the Third Strang Lemma of \cite{Di-Pietro.Droniou:18}.
The discrete pressure-velocity coupling hinges on local reconstructions of the velocity divergence and the pressure gradient obtained mimicking integration by parts formulas.
The key to the discretization of the viscous term are, on the other hand, a velocity gradient and a stabilization bilinear form.
The description of the scheme is completed by prescribing a velocity interpolator at elements.
Starting from an abstract scheme based on the above ingredients, we identify inclusion relations of the local pressure spaces into the local velocity spaces that guarantee stability in the form of a generalized inf-sup condition.
A standard inf-sup condition is recovered when the discrete velocity divergence and pressure gradient satisfy a global discrete integration by parts formula for interpolates of smooth velocity fields.
Under proper choices of discrete spaces and operators, we derive error estimates that distinguish the velocity- and pressure-contributions to the error.
Such error estimates turn out to be pressure-robust when standard inf-sup stability holds, as well as suitable inclusion relations between the local velocity and pressure spaces.
This abstract framework is applied to derive new pressure-robust error estimates for the classical Botti--Massa scheme of \cite{Botti.Massa:22}.
A similar analysis of the Rhebergen--Wells method of \cite{Rhebergen.Wells:18} is also carried out, providing alternative proofs of the findings of \cite{Kirk.Rhebergen:19}.
Finally, new schemes are also identified and analyzed, both on standard and polyhedral meshes, and extensive numerical validation of the theoretical properties is provided.

The importance of deriving velocity error estimates independent of the pressure has been pointed out in several works; see, e.g., \cite{Linke.Merdon:16, Lederer.Linke.ea:17,Kreuzer.Zanotti:20}. A classical strategy for achieving pressure-robustness consists in using stable mixed methods with $H(\operatorname{div})$-conforming and divergence-free approximate velocities \cite{John.Linke.ea:17}. This strategy has been, in particular, pursued in the context of Discontinuous Galerkin methods \cite{Cockburn.Kanshat.ea:07,Wang.Ye:07,Kreuzer.Verfurth.ea:21}.
Recently, it has also been suggested in \cite{Kreuzer.Zanotti:20,Kreuzer.Verfurth.ea:21} that pressure-robust techniques can be used to remedy the suboptimal approximation results for non-Newtonian Stokes problems \cite{Belenki.Berselli.ea:12,Botti.Castanon-Quiroz.ea:21}.
Pressure-robust methods supporting general polyhedral meshes and based on local $H(\operatorname{div})$-conforming reconstructions on simplicial submeshes have been explored, e.g., in \cite{Frerichs.Merdon:22,Castanon-Quiroz.Di-Pietro:24}.
Recent works have also pointed out that pressure-robustness on polytopal meshes can be achieved without using a submesh when compatible approximation of the curl-curl formulation of the Navier--Stokes equations are used \cite{Beirao-da-Veiga.Dassi.ea:22,Di-Pietro.Droniou.ea:24}.

The rest of this paper is organized as follows.
In Section~\ref{sec:hybrid.spaces} we set the stage of the paper by introducing hybrid spaces, interpolators, and reconstructions of quantities in terms of the hybrid unknowns.
In Section~\ref{sec:stokes} we state an abstract discretization method for the Stokes problem and prove (generalized) inf-sup stability of the pressure-velocity coupling.
A full stability and convergence analysis of the scheme is provided in Section~\ref{sec:analysis}.
Applications of the framework to existing and new schemes are considered in Section~\ref{sec:applications}.
Finally, a numerical assessment of the predicted convergence rates and pressure-robustness properties is provided in Section~\ref{sec:numerical.examples}.


\section{Hybrid spaces, interpolators, and reconstructions}\label{sec:hybrid.spaces}

\subsection{Mesh and notation for inequalities}

Let $\Omega \subset \mathbb{R}^d$, $d \ge 1$, denote an open bounded polytopal domain.
Denote by $\mathcal{M}_h = (\mathcal{T}_h, \mathcal{F}_h)$ a mesh of the domain $\Omega$ in the sense of \cite[Chapter~1]{Di-Pietro.Droniou:20}, with $\mathcal{T}_h$ collecting the mesh elements and $\mathcal{F}_h$ the mesh faces.
Notice that the term faces refers to (hyper-)planar portions of the element boundaries, i.e., faces when $d = 3$, edges when $d = 2$.
Additional assumptions on the shape of the elements may be needed for specific methods in Section~\ref{sec:applications}.
The set of faces is partitioned into the sets $\mathcal{F}_h^{\rm i}$ of internal faces and $\mathcal{F}_h^{\rm b}$ of boundary faces.
For each mesh element $T \in \mathcal{T}_h$, we denote by $\mathcal{F}_T$ the set of faces that lie on its boundary $\partial T$ and, for any $F \in \mathcal{F}_T$, we denote by $n_{TF}$ the unit vector normal to $F$ pointing out of $T$.
For each internal face $F \in \mathcal{F}_h^{\rm i}$, we additionally fix once and for all an orientation through the unit normal vector $n_F$.
For boundary faces, we select $n_F$ pointing out of $\Omega$.

When stating convergence results, we assume that $\mathcal{M}_h$ belongs to a sequence of refined meshes indexed by the mesh size $h$.
Inequalities that hold up to a multiplicative constant independent of $h$ are denoted with the $\lesssim$ sign, and we furthermore write $a \simeq b$ in place of ``$a \lesssim b$ and $b \lesssim a$''.

\subsection{Hybrid velocity and pressure spaces}

We consider methods where the velocity and pressure are both approximated using hybrid spaces, i.e., spaces spanned by vectors of local functions attached to both element and faces:
\[
\underline{U}_h \coloneq
\left(
\bigtimes_{T \in \mathcal{T}_h} U_T
\right) \times \left(
\bigtimes_{F \in \mathcal{F}_h} U_F
\right),
\qquad
\underline{P}_h \coloneq
\left(
\bigtimes_{T \in \mathcal{T}_h} P_T
\right) \times \left(
\bigtimes_{F \in \mathcal{F}_h} P_F
\right),
\]
where $U_T \subset L^2(T)^d$ and $P_T \subset L^2(T)$ for all $T \in \mathcal{T}_h$ with $P_T$ containing at least constant functions on $T$,
while $U_F \subset L^2(F)^d$ and $P_F \subset L^2(F)$ for all $F \in \mathcal{F}_h$.
Following standard conventions in the framework of fully discrete methods, we denote the restrictions of the above spaces and their elements to a mesh element $T \in \mathcal{T}_h$ by replacing the subscript $h$ with $T$.

\subsection{Interpolators and projectors}

We denote by $I_{U,T} : H^1(T)^d \to U_T$ the interpolator onto $U_T$ and by $I_{U,F} : L^2(F)^d \to U_F$ the $L^2$-orthogonal projector onto $U_F$, and set, for all $v \in H^1(\Omega)^3$,
\[
\underline{I}_{U,h} v
\coloneq \big(
(I_{U,T} v)_{T \in \mathcal{T}_h}, (I_{U,F} v)_{F \in \mathcal{F}_h}
\big).
\]
Similarly, we denote by $I_{P,Y}:L^2(Y) \to P_Y$, $Y \in \mathcal{T}_h \cup \mathcal{F}_h$, the $L^2$-orthogonal projector onto $P_Y$ and set, for all $q \in H^1(\Omega)$,
\[
\underline{I}_{P,h} q
\coloneq \big(
(I_{P,T} q)_{T \in \mathcal{T}_h}, (I_{P,F} q)_{F \in \mathcal{F}_h}
\big).
\]
Notice that, while $I_{U,F}$, $I_{P,T}$, and $I_{P,F}$ are all $L^2$-orthogonal projectors, we leave more freedom for $I_{U,T}$, as it will be needed in the analysis of certain methods in Section~\ref{sec:applications}.
Specifically, we only make the following continuity requirement:
For all $v \in H_0^1(\Omega)^d$,
\begin{equation}\label{eq:IUh:continuity}
  \| \underline{I}_{U,h} v \|_{1,h} \lesssim \| v \|_{H^1(\Omega)^d},
\end{equation}
where the $H^1$-like seminorm $\| \cdot \|_{1,h}$ on $\underline{U}_h$ is such that, for all $\underline{v}_h \in \underline{U}_h$,
\begin{equation}\label{eq:norm.1h}
  \begin{gathered}
    \| \underline{v}_h \|_{1,h}^2
    \coloneq \sum_{T \in \mathcal{T}_h} \| \underline{v}_T \|_{1,T}^2
    \\
    \text{%
    with
    $\| \underline{v}_T \|_{1,T}^2
    \coloneq \| \nabla v_T \|_{L^2(T)^{d\times d}}^2
    + h_T^{-1} \sum_{F \in \mathcal{F}_T} \| v_F - v_T \|_{L^2(F)^d}^2$
    for all $T \in \mathcal{T}_h$.      
    }
    \end{gathered}
\end{equation}
The $L^2$-orthogonal projector on a space $X$ different from the component spaces will be denoted by $\pi_X$.

\subsection{Discrete velocity divergence}

Let the discrete velocity divergence $D_T : \underline{U}_T \to P_T$ be such that, for all $\underline{v}_T \in \underline{U}_T$,
\begin{equation}\label{eq:DT}
  \int_T D_T \underline{v}_T\,q_T
  = -\int_T v_T \cdot \nabla q_T
  + \sum_{F \in \mathcal{F}_T} \int_F (v_F \cdot n_{TF})\, q_T
  \qquad\forall q_T \in P_T.
\end{equation}

\begin{assumption}\label{ass:PT.vs.UT}
  For all $T \in \mathcal{T}_h$, it holds
  \begin{enumerate}
  \item $\nabla P_T \subset U_T$ and $\pi_{\nabla P_T} \circ I_{U,T} = \pi_{\nabla P_T}$ with $\pi_{\nabla P_T}$ $L^2$-orthogonal projector onto $\nabla P_T$;
  \item $\tr_F P_T \subset U_F \cdot n_{TF}$ for all $F \in \mathcal{F}_T$.
  \end{enumerate}
\end{assumption}

\begin{proposition}[Commutativity of $D_T$]\label{prop:DT:commutativity}
  Under Assumption~\ref{ass:PT.vs.UT} we have, for all $v \in H^1(T)^d$,
  \begin{equation}\label{eq:DT:commutativity}
    D_T \underline{I}_{U,T} v = I_{P,T} ( \nabla \cdot v ).
  \end{equation}
\end{proposition}

\begin{proof}
  Apply the definition \eqref{eq:DT} of $D_T$ to $\underline{v}_T = \underline{I}_{U,T} v$, notice that, by Assumption~\ref{ass:PT.vs.UT}, $I_{U,T}$ and $I_{U,F}$ can be cancelled from the right-hand side, and integrate the latter by parts.
\end{proof}

\subsection{Discrete pressure gradient}

We next define the pressure gradient $G_T : \underline{P}_T \to U_T$ such that, for all $\underline{q}_T \in \underline{P}_T$,
\begin{equation}\label{eq:GT}
  \int_T G_T \underline{q}_T \cdot v_T
  = - \int_T q_T\, (\nabla \cdot v_T)
  + \sum_{F \in \mathcal{F}_T} \int_F q_F\, (v_T \cdot n_{TF})
  \qquad\forall v_T \in U_T.
\end{equation}

\begin{assumption}\label{ass:UT.vs.PT}
  For all $T \in \mathcal{T}_h$, it holds
  \begin{enumerate}
  \item $\nabla \cdot U_T \subset P_T$;
  \item $U_F \cdot n_{TF} \subset P_F$ for all $F \in \mathcal{F}_T$.
  \end{enumerate}
\end{assumption}

The proof of the following result is analogous to that of Proposition~\ref{prop:DT:commutativity}.

\begin{proposition}[Commutativity of $G_T$]\label{prop:GT:commutativity}
  Under Assumption~\ref{ass:UT.vs.PT} we have, for all $T \in \mathcal{T}_h$ and all $q \in H^1(T)$,
  \begin{equation}\label{eq:GT:commutativity}
    G_T \underline{I}_{P,T} q = \pi_{U_T} (\nabla q).
  \end{equation}
\end{proposition}

For future use we note the following formula, which establishes a link between the discrete velocity divergence and pressure gradient:
For all $(\underline{v}_T, \underline{q}_T) \in \underline{U}_T \times \underline{P}_T$,
\begin{equation}\label{eq:ibp}
  \begin{aligned}
    \int_T G_T \underline{q}_T \cdot v_T
    &= \int_T \nabla q_T \cdot v_T
    + \sum_{F \in \mathcal{F}_T}\int_F (q_F - q_T) (v_T \cdot n_{TF})
    \\
    \overset{\eqref{eq:DT}}&= -\int_T D_T \underline{v}_T\,q_T
    + \sum_{F \in \mathcal{F}_T} \int_F (v_T - v_F) \cdot n_{TF}\, (q_F - q_T),
  \end{aligned}
\end{equation}
where the equality in the first line is obtained integrating by parts the right-hand side of \eqref{eq:GT}.


\section{Hybrid discretizations of the Stokes problem}\label{sec:stokes}

In this section we formulate an abstract approximation scheme for the Stokes problem:
Given $f : \Omega \to \mathbb{R}^d$, find $u : \Omega \to \mathbb{R}^d$ and $p: \Omega \to \mathbb{R}$ such that
\begin{equation}\label{eq:stokes}
  \begin{alignedat}{2}
    - \nu \Delta u + \nabla p &= f &\qquad& \text{in $\Omega$},
    \\
    \nabla \cdot u &= 0 &\qquad& \text{in $\Omega$},
    \\
    u &= 0 &\qquad& \text{on $\partial \Omega$},
    \\
    \int_\Omega p &= 0,
  \end{alignedat}
\end{equation}
where $\nu > 0$ denotes the kinematic viscosity.
Throughout the rest of this work we assume $f \in L^2(\Omega)^d$ and focus on the standard weak formulation of problem \eqref{eq:stokes} with velocity $u \in H^1_0(\Omega)^d$ and pressure $p \in L^2_0(\Omega) \coloneq \left\{ q \in L^2(\Omega)\,:\, \int_\Omega q = 0\right\}$.
The velocity approximation will be sought in the following subspace of $\underline{U}_h$ incorporating homogeneous Dirichlet boundary conditions:
\[
\underline{U}_{h,0} \coloneq \left\{
\big( (v_T)_{T \in \mathcal{T}_h}, (v_F)_{F \in \mathcal{F}_h} \big)
\in \underline{U}_h \,:\, \text{$v_F = 0$ for all $F \in \mathcal{F}_h^{\rm b}$}
\right\}.
\]
To define the discrete space for the pressure, for all $\underline{q}_h \in \underline{P}_h$ we let $q_h \in L^2(\Omega)$ (not underlined) be such that $(q_h)_{|T} \coloneq q_T$ for all $T \in \mathcal{T}_h$ and set
\[
\underline{P}_{h,0} \coloneq \left\{ \underline{q}_h \in \underline{P}_h \,:\, \int_\Omega q_h = 0\right\}.
\]

\subsection{Pressure-velocity coupling}

Let $b_h : \underline{U}_h \times \underline{P}_h \to \mathbb{R}$ be such that, for all $(\underline{v}_h, \underline{q}_h) \in \underline{U}_h \times \underline{P}_h$,
\begin{equation}\label{eq:bh}
  \text{%
    $b_h(\underline{v}_h, \underline{q}_h)
    \coloneq \sum_{T \in \mathcal{T}_h} b_T(\underline{v}_T, \underline{q}_T)$
    with 
    $b_T(\underline{v}_T, \underline{q}_T) \coloneq \int_T G_T \underline{q}_T \cdot v_T$.
  }
\end{equation}
Notice that, by \eqref{eq:ibp}, the following equivalent reformulation of $b_h$ holds:
\begin{equation}\label{eq:bh:bis}
  b_h(\underline{v}_h, \underline{q}_h)
  = -\sum_{T \in \mathcal{T}_h} \int_T D_T \underline{v}_T\, q_T
  + \sum_{T \in \mathcal{T}_h} \sum_{F \in \mathcal{F}_T} \int_F (v_T - v_F) \cdot n_{TF}\,(q_F - q_T).
\end{equation}

Let us now consider the stability of the pressure-velocity coupling.
To this purpose, we define the boundary pressure seminorm $|\cdot|_{0,h} : \underline{P}_h \to \mathbb{R}$ such that, for all $\underline{q}_h \in \underline{P}_h$,
\begin{equation}\label{eq:seminorm.0h}
  \text{%
    $| \underline{q}_h |_{0,h}^2
    \coloneq \sum_{T \in \mathcal{T}_h} | \underline{q}_T |_{0,T}^2$
    with
    $| \underline{q}_T |_{0,T}^2
    \coloneq h_T \sum_{F \in \mathcal{F}_T} \| q_F - q_T \|_{L^2(F)}^2$
    for all $T \in \mathcal{T}_h$.
  }
\end{equation}

\begin{assumption}\label{ass:global.ibp}
  For all $(v, \underline{q}_h) \in H^1(\Omega)^d \times \underline{P}_h$, it holds
  \begin{equation}\label{eq:global.ibp}
    \mathcal{E}_{{\rm ibp},h}(v, \underline{q}_h)
    \coloneq
    \sum_{T \in \mathcal{T}_h} \sum_{F \in \mathcal{F}_T} \int_F (I_{U,T} v - I_{U,F} v)\cdot n_{TF}\,(q_F - q_T) = 0.
  \end{equation}
  Throughout the rest of the paper, we let
  \[
  \delta \coloneq \begin{cases}
    0 & \text{if Assumption~\ref{ass:global.ibp} holds,} \\
    1 & \text{otherwise.}
  \end{cases}
  \]
\end{assumption}

Notice that the quantity $\mathcal{E}_{{\rm ibp},h}(v, \underline{q}_h)$ measures the failure to satisfy a global integration by parts formula, and therefore corresponds a conformity error in the usual finite element terminology.

\begin{lemma}[Inf-sup and generalized inf-sup conditions]
  Let Assumption~\ref{ass:PT.vs.UT} hold and set, for all $\underline{q}_h \in \underline{P}_h$,
  \[
  \| \underline{q}_h \|_{{\rm P},h} \coloneq \left(
  \| q_h \|_{L^2(\Omega)}^2
  + \sum_{T \in \mathcal{T}_h} h_T^2 \| G_T \underline{q}_T \|_{L^2(T)^d}^2
  \right)^{\frac12}.
  \]
  Then, for all $\underline{q}_h \in \underline{P}_{h,0}$, we have the following (generalized) inf-sup condition:
  \begin{equation}\label{eq:bh:inf-sup}
    \| \underline{q}_h \|_{{\rm P},h}
    \lesssim \sup_{\underline{v}_h \in \underline{U}_{h,0} \setminus \{ \underline{0} \}}
    \frac{b_h(\underline{v}_h, \underline{q}_h)}{\| \underline{v}_h \|_{1,h}}
    + \delta | \underline{q}_h |_{0,h},
  \end{equation}
  which reduces to a standard inf-sup condition when Assumption~\ref{ass:global.ibp} holds.
\end{lemma}

\begin{proof}
  Let $\underline{q}_h \in \underline{P}_{h,0}$ and, for the sake of brevity, denote by $\$$ the supremum in the right-hand side of \eqref{eq:bh:inf-sup}.
  Since $q_h \in L^2_0(\Omega)$, the continuous inf-sup condition yields the existence of
  \begin{equation}\label{eq:continuous.inf-sup}
    \text{%
      $v_q \in H_0^1(\Omega)^d$ such that $\nabla \cdot v_q = q_h$ and $\| v_q \|_{H^1(\Omega)^d} \lesssim \| q_h \|_{L^2(\Omega)}$.
    }
  \end{equation}
  Let $\underline{w}_h \in \underline{U}_h$ be such that $w_T = h_T^2\, G_T \underline{q}_T$ for all $T \in \mathcal{T}_h$ and $w_F = 0$ for all $F \in \mathcal{F}_h$ and notice that, by discrete inverse and trace inequalities,
  \begin{equation}\label{eq:wh:norm.equivalence}
    \| \underline{w}_h \|_{1,h}
    \lesssim \left(
    \sum_{T \in \mathcal{T}_h} h_T^2 \| G_T \underline{q}_T \|_{L^2(T)^d}^2
    \right)^{\frac12}.
  \end{equation}
  We have
  \[
  \begin{aligned}
    \| \underline{q}_h \|_{{\rm P},h}^2
    \overset{\eqref{eq:continuous.inf-sup},\,\eqref{eq:bh}}&= \int_\Omega (\nabla \cdot v_q)\,q_h
    + b_h(\underline{w}_h, \underline{q}_h)
    \\
    &= \sum_{T \in \mathcal{T}_h} \int_\Omega  I_{P,T}(\nabla \cdot v_q)\,q_T
    + b_h(\underline{w}_h, \underline{q}_h)
    \overset{\eqref{eq:DT:commutativity}}= \sum_{T \in \mathcal{T}_h} \int_\Omega (D_T \underline{I}_{U,T} v_q)\,q_T
    + b_h(\underline{w}_h, \underline{q}_h)
    \\
    \overset{\eqref{eq:bh:bis}}&=
    -b_h(\underline{I}_{U,h} v_q, \underline{q}_h)
    + \delta \sum_{T \in \mathcal{T}_h} \sum_{F \in \mathcal{F}_T} \int_F (I_{U,T} v_q - I_{U,F} v_q)\cdot n_{TF}\,(q_F - q_T)
    + b_h(\underline{w}_h, \underline{q}_h)
    \\
    &\le \$ \| \underline{I}_{U,h} v_q \|_{1,h}
    + \delta \| \underline{I}_{U,h} v_q \|_{1,h} | \underline{q}_h |_{0,h}
    + \$ \| \underline{w}_h \|_{1,h}
    \\
    \overset{\eqref{eq:IUh:continuity}}&\lesssim \left(
    \$ + \delta | \underline{q}_h |_{0,h}
    \right) \| v_q \|_{H^1(\Omega)^d}
    + \$ \| \underline{w}_h \|_{1,h}
    \overset{\eqref{eq:continuous.inf-sup},\,\eqref{eq:wh:norm.equivalence}}\lesssim
    \left(
    \$ + \delta | \underline{q}_h |_{0,h}
    \right) \| \underline{q}_h \|_{{\rm P},h},
  \end{aligned}
  \]
  where, in the first inequality, we have used the definition of supremum for the first and third terms and Cauchy--Schwarz inequalities along with the definitions \eqref{eq:norm.1h} of $\|\cdot\|_{1,h}$ and \eqref{eq:seminorm.0h} of $|\cdot|_{0,h}$ for the second term.
  Simplifying, the conclusion follows.
\end{proof}

\subsection{Viscous terms}

Let a local space $\Sigma_T \subset L^2(T)^{d \times d}$ be given.
The discretization of the viscous terms is based on the velocity gradient $E_T : \underline{U}_T \to \Sigma_T$ such that, for all $\underline{v}_T \in \underline{U}_T$ and all $\tau \in \Sigma_T$,
\begin{subequations}
  \begin{align}\label{eq:ET}
    \int_T E_T \underline{v}_T : \tau
    &= - \int_T v_T \cdot (\nabla \cdot \tau)
    + \sum_{F \in \mathcal{F}_T} \int_F v_F \cdot (\tau n_{TF})
    \\ \label{eq:ET'}
    &= \int_T \nabla v_T : \tau
    + \sum_{F \in \mathcal{F}_T} \int_F (v_F - v_T) \cdot (\tau n_{TF}).
  \end{align}
\end{subequations}
The space $\Sigma_T$ is typically selected in order for the following assumption to hold true.

\begin{assumption}\label{ass:SigmaT}
  The following holds:
  \begin{enumerate}
  \item $\nabla \cdot \Sigma_T \subset U_T$ and $\pi_{\nabla \cdot \Sigma_T} \circ I_{U,T} = \pi_{\nabla \cdot \Sigma_T}$;
  \item $\Sigma_T n_{TF} \subset U_F$ for all $F \in \mathcal{F}_T$.
  \end{enumerate}
\end{assumption}

The proof of the following result is analogous to that of Proposition~\ref{prop:DT:commutativity}.

\begin{proposition}[Commutativity of $E_T$]\label{prop:ET:commutativity}
  Under Assumption~\ref{ass:SigmaT}, it holds, for all $v \in H^1(T)^d$,
  \[
    E_T \underline{I}_{U,T} v = \pi_{\Sigma_T} ( \nabla v).
  \]
\end{proposition}

We let $a_h : \underline{U}_h \times \underline{U}_h \to \mathbb{R}$ be such that, for all $(\underline{w}_h, \underline{v}_h) \in \underline{U}_h \times \underline{U}_h$,
\begin{equation}\label{eq:ah}
  \text{%
    $a_h(\underline{w}_h, \underline{v}_h)
    \coloneq \sum_{T \in \mathcal{T}_h} a_T(\underline{w}_T, \underline{v}_T)$
    with $a_T(\underline{w}_T, \underline{v}_T) \coloneq
    \int_T E_T \underline{w}_T : E_T \underline{v}_T + s_T(\underline{w}_T, \underline{v}_T)$,
  }
\end{equation}
where $s_T : \underline{U}_T \times \underline{U}_T \to \mathbb{R}$ is a (possibly zero) symmetric positive semi-definite stabilization bilinear form.

\begin{assumption}\label{ass:sT}
  Recalling \eqref{eq:norm.1h}, the following uniform norm equivalence holds:
  \[
  \| \underline{v}_h \|_{1,h} \simeq a_h(\underline{v}_h, \underline{v}_h)^{\frac12}\qquad
  \forall \underline{v}_h \in \underline{U}_{h,0}.
  \]
\end{assumption}

\subsection{An abstract hybrid scheme}

We consider the following hybrid discretization of problem \eqref{eq:stokes}:
Find $(\underline{u}_h, \underline{p}_h) \in \underline{U}_{h,0} \times \underline{P}_{h,0}$ such that
\begin{equation}\label{eq:discrete}
  \begin{aligned}
    \nu a_h(\underline{u}_h, \underline{v}_h)
    + b_h(\underline{v}_h, \underline{p}_h)
    &= \int_\Omega f \cdot v_h
    &\qquad& \forall \underline{v}_h \in \underline{U}_{h,0},
    \\
    -b_h(\underline{u}_h, \underline{q}_h)
    + \delta \nu^{-1} d_h(\underline{p}_h, \underline{q}_h)
    &= 0
    &\qquad& \forall \underline{q}_h \in \underline{P}_{h,0},
  \end{aligned}
\end{equation}
where $v_h \in L^2(\Omega)^d$ (not underlined) is such that $(v_h)_{|T} \coloneq v_T$ for all $T \in \mathcal{T}_h$ and the bilinear form $d_h : \underline{P}_h \times \underline{P}_h \to \mathbb{R}$ is such that
\[
\text{%
  $d_h(\underline{p}_h, \underline{q}_h)
  \coloneq \sum_{T \in \mathcal{T}_h} d_T(\underline{p}_T, \underline{q}_T)$
  with $d_T(\underline{p}_T, \underline{q}_T) \coloneq h_T \sum_{F \in \mathcal{F}_T} \int_F (p_F - p_T)\,(q_F - q_T)$
  for all $T \in \mathcal{T}_h$.
}
\]
Defining the bilinear form $\mathcal{A}_h : \left[ \underline{U}_h \times \underline{P}_h \right]^2 \to \mathbb{R}$ such that, for all $( (\underline{w}_h, \underline{r}_h), (\underline{v}_h, \underline{p}_h) ) \in \left[ \underline{U}_h \times \underline{P}_h \right]^2$,
\begin{equation}\label{eq:Ah}
  \mathcal{A}_h((\underline{w}_h, \underline{r}_h), (\underline{v}_h, \underline{q}_h))
  \coloneq
  \nu a_h(\underline{w}_h, \underline{v}_h)
  + b_h(\underline{v}_h, \underline{r}_h)
  - b_h(\underline{w}_h, \underline{q}_h)
  + \delta \nu^{-1} d_h(\underline{r}_h, \underline{q}_h),
\end{equation}
problem \eqref{eq:discrete} can be equivalently reformulated as:
Find $(\underline{u}_h, \underline{p}_h) \in \underline{U}_{h,0} \times \underline{P}_{h,0}$ such that
\begin{equation}\label{eq:discrete:bis}
  \mathcal{A}_h((\underline{u}_h, \underline{p}_h), (\underline{v}_h, \underline{q}_h))
  = \int_\Omega f \cdot v_h
  \qquad\forall (\underline{v}_h, \underline{q}_h) \in \underline{U}_{h,0} \times \underline{P}_{h,0}.
\end{equation}


\section{Analysis}\label{sec:analysis}

This section contains the stability and error analysis of the scheme \eqref{eq:discrete}.
A summary of Assumptions~\ref{ass:PT.vs.UT}--\ref{ass:sT} and of their respective roles in the analysis is provided in Table~\ref{tab:assumptions+roles}.

\begin{table}\centering
  \renewcommand{\arraystretch}{1.2}
  \begin{minipage}{\textwidth}\centering
      \begin{tabular}{cc}
        \toprule
        Reference
        & Assumption \\
        \midrule
        Assumption \ref{ass:PT.vs.UT} &
        for all $T \in \mathcal{T}_h$,\quad
        \begin{minipage}{7cm}
          1. $\nabla P_T \subset U_T$ and $\pi_{\nabla P_T} \circ I_{U,T} = \pi_{\nabla P_T}$
          \\
          2. $\tr_F P_T \subset U_F \cdot n_{TF}$ for all $F \in \mathcal{F}_T$
        \end{minipage}
        \\ \midrule
        Assumption \ref{ass:UT.vs.PT} &
        for all $T \in \mathcal{T}_h$,\quad
        \begin{minipage}{7cm}
          1. $\nabla \cdot U_T \subset P_T$
          \\
          2. $U_F \cdot n_{TF} \subset P_F$ for all $F \in \mathcal{F}_T$
        \end{minipage}
        \\ \midrule
        Assumption \ref{ass:global.ibp}
        &
        for all $(v, \underline{q}_h) \in H^1(\Omega)^d \times \underline{P}_h$,\quad
        $\mathcal{E}_{{\rm ibp},h}(v, \underline{q}_h) = 0$
        \\ \midrule
        Assumption \ref{ass:SigmaT}
        &
        for all $T \in \mathcal{T}_h$,\quad
        \begin{minipage}{7cm}
          1. $\nabla \cdot \Sigma_T \subset U_T$ and $\pi_{\nabla \cdot \Sigma_T} \circ I_{U,T} = \pi_{\nabla \cdot \Sigma_T}$
          \\
          2. $\Sigma_T n_{TF} \subset U_F$ for all $F \in \mathcal{F}_T$
        \end{minipage}
        \\ \midrule
        Assumption \ref{ass:sT}
        &
        $\| \cdot \|_{1,h}
        \simeq a_h(\cdot, \cdot)^{\frac12}$
        \\
        \bottomrule
      \end{tabular}
    \subcaption{Summary of the assumptions used for the analysis.\label{tab:assumptions}}
  \end{minipage}
  \medskip\\
  \begin{minipage}{\textwidth}\centering
    \begin{tabular}{cccc}
      \toprule
      Reference & Result & Required assumptions & Optional assumptions: consequences \\
      \midrule
      Lemma~\ref{lem:stability} & Stability & \ref{ass:PT.vs.UT} and \ref{ass:sT} & \ref{ass:global.ibp}: No need for pressure stabilization \\
      Lemma \ref{lem:err.est} & Error estimate & \ref{ass:PT.vs.UT}, \ref{ass:SigmaT}, and \ref{ass:sT} & \ref{ass:UT.vs.PT}, \ref{ass:global.ibp}: Pressure-robustness \\
      \bottomrule
    \end{tabular}
    \subcaption{Roles of the assumptions in the analysis of the scheme \eqref{eq:discrete}.\label{tab:role.assumptions}}
  \end{minipage}
  \caption{Summary and roles of Assumptions \ref{ass:PT.vs.UT}--\ref{ass:sT}.\label{tab:assumptions+roles}}
\end{table}

\subsection{Stability}

The norm used for the analysis is $\|\cdot\|_{\nu,h} : \underline{U}_h \times \underline{P}_h$ such that, for all $(\underline{v}_h, \underline{q}_h) \in \underline{U}_h \times \underline{P}_h$,
\begin{equation}\label{eq:norm.h}
  \| (\underline{v}_h, \underline{q}_h) \|_{\nu,h}
  \coloneq \left(
  \nu \| \underline{v}_h \|_{1,h}^2
  + \nu^{-1} \| \underline{q}_h \|_{{\rm P},h}^2
  + \delta \nu^{-1} | \underline{q}_h |_{0,h}^2
  \right)^{\frac12}.
\end{equation}
This map defines a norm on $\underline{U}_{h,0}\times \underline{P}_{h,0}$ when $\delta = 1$.
We assume that this is also the case when $\delta = 0$.
We note the following discrete Poincar\'e inequality on hybrid spaces, which can be proved reasoning as in \cite[Lemma~2.15]{Di-Pietro.Droniou:20}:
\begin{equation}\label{eq:poincare}
  \| v_h \|_{L^2(\Omega)^d} \lesssim \| \underline{v}_h \|_{1,h} \qquad \forall \underline{v}_h \in \underline{U}_{h,0}.
\end{equation}

\begin{lemma}[Stability of the scheme]\label{lem:stability}
  Under Assumptions~\ref{ass:PT.vs.UT} and~\ref{ass:sT}, the following uniform inf-sup condition holds:
  For all $(\underline{v}_h, \underline{q}_h) \in \underline{U}_{h,0} \times \underline{P}_{h,0}$,
  \begin{equation}\label{eq:inf-sup}
    \| (\underline{w}_h, \underline{r}_h) \|_{\nu,h}
    \lesssim \sup_{(\underline{v}_h, \underline{q}_h) \in \underline{U}_{h,0} \times \underline{P}_{h,0} \setminus \{ (\underline{0}, \underline{0}) \}}
    \frac{%
      \mathcal{A}_h((\underline{w}_h, \underline{r}_h),(\underline{v}_h, \underline{q}_h))
    }{%
      \| (\underline{v}_h, \underline{q}_h) \|_{\nu,h}
    }.
  \end{equation}
  Hence, problem \eqref{eq:discrete} (or, equivalently, \eqref{eq:discrete:bis}) admits a unique solution which satisfies the following a priori estimate:
  \begin{equation}\label{eq:a-priori}
    \| (\underline{u}_h, \underline{p}_h) \|_{\nu,h} \lesssim
    \nu^{-\frac12} \| f \|_{L^2(\Omega)^d}.
  \end{equation}
\end{lemma}

\begin{remark}[A priori estimate]
  Following, e.g., \cite[Section~2.3]{Castanon-Quiroz.Di-Pietro:20}, in the right-hand side of \eqref{eq:a-priori} we could replace $f$ with the irrotational part of the forcing term.
\end{remark}

\begin{proof}[Proof of Lemma~\ref{lem:stability}]
  Let $(\underline{w}_h, \underline{r}_h) \in \underline{U}_{h,0} \times \underline{P}_{h,0}$ and denote by $\$$ the supremum in the right-hand side of \eqref{eq:inf-sup}.
  Taking $(\underline{v}_h, \underline{q}_h) = (\underline{w}_h, \underline{r}_h)$ in \eqref{eq:Ah} and using Assumption~\ref{ass:sT}, we get
  \begin{equation}\label{eq:inf-sup:coercivity}
    \nu \| \underline{w}_h \|_{1,h}^2
    + \delta \nu^{-1} | \underline{r}_h |_{0,h}^2
    \lesssim
    \mathcal{A}((\underline{w}_h, \underline{r}_h), (\underline{w}_h, \underline{r}_h))
    \lesssim \$ \| (\underline{w}_h, \underline{r}_h) \|_{\nu,h}.
  \end{equation}
  By the (generalized) inf-sup condition \eqref{eq:bh:inf-sup} on $b_h$ we have, on the other hand,
  \begin{equation}\label{eq:inf-sup:L2.pressure}
    \begin{aligned}
      \nu^{-\frac12} \| \underline{r}_h \|_{{\rm P},h}
      &\lesssim
       \nu^{-\frac12} \sup_{\underline{v}_h \in \underline{U}_{h,0} \setminus \{ \underline{0} \}} \frac{ b_h(\underline{v}_h, \underline{r}_h) }{  \| \underline{v}_h \|_{1,h} }
      + \delta \nu^{-\frac12} | \underline{r}_h |_{0,h}
      \\
      \overset{\eqref{eq:Ah}}&=
      \sup_{\underline{v}_h \in \underline{U}_{h,0} \setminus \{ \underline{0} \}} \frac{ \mathcal{A}_h(\underline{w}_h, \underline{r}_h), (\underline{v}_h, \underline{0})) - \nu a_h(\underline{w}_h, \underline{v}_h) }{ \nu^{\frac12} \| \underline{v}_h \|_{1,h} }
      + \delta \nu^{-\frac12} | \underline{r}_h |_{0,h}
      \\
      &\lesssim
      \$
      + \nu^{\frac12} \| \underline{w}_h \|_{1,h}
      + \delta \nu^{-\frac12} | \underline{r}_h |_{0,h}
      \overset{\eqref{eq:inf-sup:coercivity}}\lesssim
      \$ + \$^{\frac12}\| (\underline{w}_h, \underline{r}_h) \|_{\nu,h}^{\frac12},
    \end{aligned}
  \end{equation}
  where we have used Assumption~\ref{ass:sT} along with the symmetry of $a_h$ to write $a_h(\underline{w}_h, \underline{v}_h) \lesssim \| \underline{w}_h \|_{1,h} \| \underline{v}_h \|_{1,h}$ in order to pass to the third line.
  Squaring \eqref{eq:inf-sup:L2.pressure},
  using the fact that $(\alpha + \beta)^2 \le 2\alpha^2 + 2\beta^2$ for all real numbers $\alpha$ and $\beta$ to bound the right-hand side,
  and summing the resulting inequality to \eqref{eq:inf-sup:coercivity}, we obtain
  \[
  \| (\underline{w}_h, \underline{r}_h) \|_{\nu,h}^2
  \lesssim \$ \| (\underline{w}_h, \underline{r}_h) \|_{\nu,h} + \$^2
  \le C_\epsilon \$^2 + \epsilon \| (\underline{w}_h, \underline{r}_h) \|_{\nu,h}^2,
  \]
  where the conclusion follows from the generalized Young's inequality $ab \le \epsilon a^2 + C_\epsilon b^2$ valid for all $\epsilon > 0$.
  Taking $\epsilon = \frac12$ and simplifying, \eqref{eq:inf-sup} follows.

  The a priori estimate \eqref{eq:a-priori} follows classically from this inf-sup condition along with the discrete Poincar\'e inequality \eqref{eq:poincare}.
\end{proof}

\subsection{Error estimate}

\begin{lemma}[Error estimate]\label{lem:err.est}
  Let $(\underline{u}_h, \underline{p}_h) \in \underline{U}_{h,0} \times \underline{P}_{h,0}$ solve \eqref{eq:discrete} (or, equivalently, \eqref{eq:discrete:bis}) and let the weak solution $(u,p) \in H_0^1(\Omega)^d \times L_0^2(\Omega)$ of the Stokes problem \eqref{eq:stokes} be such that $p \in H^1(\Omega)$.
  Then, under Assumptions~\ref{ass:PT.vs.UT}, \ref{ass:SigmaT}, and \ref{ass:sT}, it holds
  \begin{equation}\label{eq:err.est}
    \| (\underline{u}_h - \underline{I}_{U,h} u, \underline{p}_h - \underline{I}_{P,h}p) \|_{\nu,h}
    \lesssim \nu^{\frac12} \mathcal{E}_u + \nu^{-\frac12} \mathcal{E}_p
  \end{equation}
  with
  \begin{subequations}\label{eq:error.components}
    \begin{align}\label{eq:Eu}
      \mathcal{E}_u
      &\coloneq \left[
        \sum_{T \in \mathcal{T}_h} \left(
        h_T \| \nabla u - \pi_{\Sigma_T} (\nabla u) \|_{L^2(\partial T)^{d\times d}}^2
        + s_T(\underline{I}_{U,T} u, \underline{I}_{U,T} u)
        \right)
        + \delta \sup_{\underline{q}_h \in \underline{P}_h,\, | \underline{q}_h |_{0,h} = 1} \mathcal{E}_{{\rm ibp},h}(u,\underline{q}_h)
        \right]^{\frac12},
        \\ \label{eq:Ep}
        \mathcal{E}_p
        &\coloneq \left[
          \sum_{T \in \mathcal{T}_h} \| \pi_{U_T} (\nabla p) - G_T \underline{I}_{P,T} p \|_{L^2(T)^d}^2
          + \delta | \underline{I}_{P,h} p |_{0,h}^2
          \right]^{\frac12}.
    \end{align}
  \end{subequations}
\end{lemma}

\begin{remark}[Pressure robustness]\label{rem:pressure.robustness}
  If Assumption~\ref{ass:UT.vs.PT} is verified, the first contribution in $\mathcal{E}_p$ vanishes by \eqref{eq:GT:commutativity}.
  On the other hand, $\delta = 0$ if Assumption~\ref{ass:global.ibp} holds, which implies that the second contribution in $\mathcal{E}_p$ is zero.
  Therefore, if both Assumptions~\ref{ass:UT.vs.PT} and \ref{ass:global.ibp} are met, recalling the definition \eqref{eq:norm.h} of $\| \cdot \|_{\nu,h}$, we infer the following error estimate for the velocity from \eqref{eq:err.est}:
  \[
  \| \underline{u}_h - \underline{I}_{U,h} u \|_{1,h}
  \lesssim \mathcal{E}_u.
  \]
  The right-hand side of \eqref{eq:err.est} does not depend on the viscosity nor on the pressure, showing that the method is pressure-robust.
\end{remark}

\begin{proof}
  Denote by $\|\cdot\|_{\nu,h,*}$ the norm adjoint to $\|\cdot\|_{\nu,h}$ in $\underline{U}_{h,0} \times \underline{P}_{h,0}$.
  By \cite[Theorem~10]{Di-Pietro.Droniou:18}, the inf-sup condition \eqref{eq:inf-sup} yields the following basic estimate:
  \begin{equation}\label{eq:err.est:basic}
    \| (\underline{u}_h - \underline{I}_{U,h} u, \underline{p}_h - \underline{I}_{P,h} p) \|_{\nu,h}
    \lesssim \| \mathcal{E}_h(\cdot,\cdot) \|_{\nu,h,*},
  \end{equation}
  where the consistency error linear form $\mathcal{E}_h : \underline{U}_{h,0} \times \underline{P}_{h,0} \to \mathbb{R}$ is such that, for all $(\underline{v}_h, \underline{q}_h) \in \underline{U}_{h,0} \times \underline{P}_{h,0}$,
  \begin{equation}\label{eq:Eh}
    \begin{aligned}
      \mathcal{E}_h(\underline{v}_h, \underline{q}_h)
      &\coloneq \int_\Omega f \cdot v_h
      - \mathcal{A}_h((\underline{I}_{U,h} u, \underline{I}_{P,h} p), (\underline{v}_h, \underline{q}_h))
      \\
      &\;=
      \underbrace{%
        \nu\left(
        -\int_\Omega \Delta u - a_h(\underline{I}_{U,h}u, \underline{v}_h)
        \right)
      }_{\mathfrak{T}_1}
      + \underbrace{%
        \int_\Omega \nabla p \cdot v_h - b_h( \underline{v}_h, \underline{I}_{P,h} p)
      }_{\mathfrak{T}_2}
      \\
      &\qquad
      + \underbrace{%
        b_h( \underline{I}_{U,h} u, \underline{q}_h)
      }_{\mathfrak{T}_3}
      + \underbrace{%
        \delta \nu^{-1} d_h(\underline{I}_{P,h} p, \underline{q}_h)
      }_{\mathfrak{T}_4},
    \end{aligned}
  \end{equation}
  where we have used the fact that $f = -\nu \Delta u + \nabla p$ almost everywhere in $\Omega$ (cf. \eqref{eq:stokes}) along with the definition \eqref{eq:Ah} of $\mathcal{A}_h$ to pass to the second line.
  We proceed to estimate the terms in the right-hand side.
  For the first term, reproducing the steps of the proof of \cite[Lemma 2.18(ii)]{Di-Pietro.Droniou:20}, we obtain
  \begin{equation}\label{eq:T1}
    \mathfrak{T}_1
    \lesssim \nu^{\frac12}\mathcal{E}_u\, \nu^{\frac12}\| \underline{v}_h \|_{1,h}
    \overset{\eqref{eq:norm.h}}\le \nu^{\frac12}\mathcal{E}_u\, \| (\underline{v}_h, \underline{q}_h) \|_{\nu,h}.
  \end{equation}

  Expanding $b_h$ according to its definition \eqref{eq:bh}, we can write for the second term:
  \begin{equation}\label{eq:T2}
    \begin{aligned}
      \mathfrak{T}_2
      &= \sum_{T \in \mathcal{T}_h} \int_T \left(
      \pi_{U_T} (\nabla p) - G_T \underline{I}_{P,T} p
      \right) \cdot v_T
      \\
      &\le
      \nu^{-\frac12}\left(
      \sum_{T \in \mathcal{T}_h} \| \pi_{U_T} (\nabla p) - G_T \underline{I}_{P,T} p \|_{L^2(T)^d}^2
      \right)^{\frac12}\,\nu^{\frac12}\| v_h \|_{L^2(\Omega)^d}
      \\
      \overset{\eqref{eq:Ep},\,\eqref{eq:poincare}}&\lesssim
      \nu^{-\frac12}\mathcal{E}_p\,\nu^{\frac12} \| \underline{v}_h \|_{1,h}
      \overset{\eqref{eq:norm.h}}\le \nu^{-\frac12}\mathcal{E}_p\, \| (\underline{v}_h, \underline{q}_h) \|_{\nu,h},
    \end{aligned}
  \end{equation}
  where the insertion of $\pi_{U_T}$ in front of $\nabla p$ in the first line is possible since $v_T \in U_T$, while to pass to the second line we have used Cauchy--Schwarz inequalities on both the integrals and the sums.

  For the third term, we have
  \begin{equation}\label{eq:T3}
    \mathfrak{T}_3
    \overset{\eqref{eq:bh:bis},\,\eqref{eq:global.ibp}}
    = -\sum_{T \in \mathcal{T}_h}  \int_T \cancel{( D_T \underline{I}_{U,T} u )}\,q_T
    + \mathcal{E}_{{\rm ibp},h}(u, \underline{q}_h)
    \overset{\eqref{eq:Eu}}\le \nu^{\frac12}\mathcal{E}_u\,\nu^{-\frac12} | \underline{q}_h |_{0,h}
    \overset{\eqref{eq:norm.h}}\le \nu^{\frac12}\mathcal{E}_u\,\| (\underline{v}_h, \underline{q}_h) \|_{\nu,h},
  \end{equation}
  where the cancellation in the first line is a consequence of the commutation property \eqref{eq:DT:commutativity} (valid since Assumption~\ref{ass:PT.vs.UT} holds).

  Finally, for the fourth term we use Cauchy--Schwarz inequalities and again the definition \eqref{eq:seminorm.0h} of $| \cdot |_{0,h}$ to write
  \begin{equation}\label{eq:T4}
    \mathfrak{T}_4
    \le \delta \nu^{-1} | \underline{I}_{P,h} p |_{0,h}\,| \underline{q}_h |_{0,h}
    \overset{\eqref{eq:Ep},\,\eqref{eq:norm.h}}
    \le \nu^{-\frac12} \mathcal{E}_p\,\| (\underline{v}_h, \underline{q}_h) \|_{\nu,h}.
  \end{equation}

  Gathering the estimates \eqref{eq:T1}, \eqref{eq:T2}, \eqref{eq:T3}, and \eqref{eq:T4} in \eqref{eq:Eh} and passing to the supremum over $(\underline{v}_h, \underline{q}_h) \in \underline{U}_{h,0} \times \underline{P}_{h,0}$ such that $\| (\underline{v}_h, \underline{q}_h) \|_{\nu,h} = 1$ gives
  \[
  \| \mathcal{E}_h(\cdot,\cdot) \|_{\nu,h,*}
  \lesssim \nu^{\frac12} \mathcal{E}_u + \nu^{-\frac12} \mathcal{E}_p.
  \]
  Plugging the above estimate into \eqref{eq:err.est:basic} yields the conclusion.
\end{proof}


\section{Applications}\label{sec:applications}

In this section we apply the framework above to existing and new schemes.

\begin{sidewaystable}\centering
  \renewcommand{\arraystretch}{1.2}
  \begin{adjustbox}{width=1\textwidth}
    \begin{tabular}{cccccccccc}
      \toprule
      Ref.
      & Mesh
      & $U_T$
      & $I_{U,T}$
      & $U_F$
      & $P_T$
      & $P_F$
      & $s_T$
      & Assumptions     
      & Conv. rate
      \\
      \midrule
      Botti--Massa \cite{Botti.Massa:22}
      & Simplicial
      & $\mathcal{BDM}^{k+1}(T)$
      & $I_{\mathcal{BDM},T}^{k+1}$
      & $\mathcal{P}^k(F)^d$
      & $\mathcal{P}^k(T)$
      & $\mathcal{P}^{k+1}(F)$
      & \eqref{eq:sT:classical.hho}
      & \ref{ass:PT.vs.UT}--\ref{ass:sT}
      & $k+1$
      \\ \midrule
      Rhebergen--Wells \cite{Rhebergen.Wells:18}
      & Simplicial
      & $\mathcal{BDM}^k(T)^d$
      & $I_{\mathcal{BDM},T}^k$
      & $\mathcal{P}^k(F)^d$
      & $\mathcal{P}^{k-1}(T)$
      & $\mathcal{P}^k(F)$
      & \eqref{eq:RW:sT}
      & \ref{ass:PT.vs.UT}-- \ref{ass:sT}
      & $k$    
      \\ \midrule
      New
      & Simplicial
      & $\mathcal{RTN}^{k+1}(T)$
      & $I_{\mathcal{RTN},T}^{k+1}$
      & $\mathcal{P}^k(F)^d$
      & $\mathcal{P}^k(T)$
      & $\mathcal{P}^k(F)$
      & \eqref{eq:sT:classical.hho}
      & \ref{ass:PT.vs.UT}-- \ref{ass:sT}
      & $k+1$
      \\ \midrule
      New
      & Cartesian
      & $\mathcal{BDFM}^{k+1}(T)$
      & $I_{\mathcal{BDFM},T}^{k+1}$
      & $\mathcal{P}^k(F)^d$
      & $\mathcal{P}^k(T)$
      & $\mathcal{P}^k(F)$
      & \eqref{eq:sT:classical.hho}
      & \ref{ass:PT.vs.UT}-- \ref{ass:sT}
      & $k+1$
      \\ \midrule
      New
      & Polytopal
      & $\mathcal{P}^{k+1}(T)^d$
      & $\pi_{\mathcal{P}^{k+1}(T)^d}$
      & $\mathcal{P}^k(F)^d$
      & $\mathcal{P}^k(T)$
      & $\mathcal{P}^k(F)$
      & \eqref{eq:sT:classical.hho}
      & \ref{ass:PT.vs.UT}, \ref{ass:UT.vs.PT}, \ref{ass:SigmaT}, and \ref{ass:sT}
      & $k+1$
      \\
      \bottomrule
    \end{tabular}
  \end{adjustbox}
  \caption{Space choices and properties.
    The last column contains the maximum convergence rate for the error norm in the left-hand side of \eqref{eq:err.est} attainable for smooth solutions when the polynomial degree $k$ is used.}
\end{sidewaystable}

\subsection{Preliminaries}

\subsubsection{Local full polynomial spaces}

For a given integer $\ell\ge 0$, we denote by $\mathbb{P}_d^\ell$ the space of $d$-variate polynomials of total degree $\le \ell$, with the convention that $\mathbb{P}_d^{-1} \coloneq \{ 0 \}$.
Given a mesh element or face $Y \in \mathcal{T}_h \cup \mathcal{F}_h$, we denote by $\mathcal{P}^\ell(Y)$ the space spanned by the restriction to $Y$ of the functions in $\mathbb{P}_d^\ell$.

\subsubsection{The Brezzi--Douglas--Marini space}

Let an integer $\ell \ge 0$ and a triangle/tetrahedron $T$ be given.
We denote by $x_T$ a point inside $T$ which, when $T$ is an element of a mesh belonging to a refined sequence, we assume at a distance $\simeq h_T$ from the boundary.
Let
\[
\mathcal{G}^\ell(T) \coloneq \nabla \mathcal{P}^{\ell+1}(T),\qquad
\mathcal{G}_{\rm c}^\ell(T) \coloneq \begin{cases}
  (x - x_T)^\perp \mathcal{P}^{\ell-1}(T) & \text{if $d = 2$},
  \\
  (x - x_T) \times \mathcal{P}^{\ell-1}(T)^3 & \text{if $d = 3$},
\end{cases}
\]
where $\perp$ indicates a rotation of $-\frac{\pi}{2}$.
For $\ell \ge 0$, the N\'ed\'elec space of the first type is
\[
\mathcal{N}^{\ell}(T)
\coloneq \mathcal{G}^{\ell-1}(T) \oplus \mathcal{G}_{\rm c}^{\ell}(T),
\]
and we adopt the convention that $\mathcal{N}^{-1}(T) \coloneq \{ 0 \}$.

For any $\ell \ge 1$, the Brezzi--Douglas--Marini space is $\mathcal{BDM}^\ell(T) \coloneq \mathcal{P}^\ell(T)^d$ equipped with the interpolator $I_{\mathcal{BDM},T}^\ell : H^1(T)^d \to \mathcal{BDM}^\ell(T)$ such that, for all $v \in H^1(T)^d$,
\begin{subequations}\label{eq:I.BDM}
  \begin{align}\label{eq:I.BDM:F}
    \forall F \in \mathcal{F}_T, &\qquad&
    \int_F (I_{\mathcal{BDM},T}^\ell v \cdot n_{TF})\, q &= \int_F (v \cdot n_{TF})\, q
    &\qquad& \forall q \in \mathcal{P}^\ell(F),
    \\ \label{eq:I.BDM:T}
      && \int_T I_{\mathcal{BDM},T}^\ell v \cdot w &= \int_T v \cdot w
      &\qquad& \forall w \in \mathcal{N}^{\ell-1}(T).
  \end{align}
\end{subequations}

\begin{remark}[$L^2$-orthogonal projections of the BDM interpolate onto full polynomial spaces]\label{eq:proj.I.BDM}
  Since $\mathcal{P}^{\ell-2}(T)^d \subset \mathcal{N}^{\ell-1}(T)$, condition \eqref{eq:I.BDM:T} implies, in particular,
  \begin{equation}\label{eq:pi.I.BDM=pi}
    \pi_{\mathcal{P}^{\ell-2}(T)^d} \circ I_{\mathcal{BDM},T}^\ell = \pi_{\mathcal{P}^{\ell-2}(T)^d}.
  \end{equation}
  We also notice that the interpolator preserves the average value, i.e., for all $\ell \ge 1$,
  \[
  \pi_{\mathcal{P}^0(T)^d} \circ I_{\mathcal{BDM},T}^\ell = \pi_{\mathcal{P}^0(T)^d}.
  \]
  For $\ell \ge 2$, this is a straightforward consequence of \eqref{eq:pi.I.BDM=pi}.
  For $\ell = 1$, we start by noticing that, for all $v \in \mathcal{BDM}^1(T)$ and all $q \in \mathcal{P}^1(T) \cap L^2_0(T)$,
  \[
  \begin{aligned}
    \int_T I_{\mathcal{BDM},T}^1 v \cdot \nabla q
    &=  - \cancel{ \int_T (\nabla \cdot I_{\mathcal{BDM},T}^1 v)\, q }
    + \sum_{F \in \mathcal{F}_T} \int_F (I_{\mathcal{BDM},T}^1 v \cdot n_{TF})\, q
    \\
    \overset{\eqref{eq:I.BDM:F}}&= \sum_{F \in \mathcal{F}_T} \int_F (v \cdot n_{TF})\, q
    = \int_F v \cdot \nabla q,
  \end{aligned}
  \]
  where the cancellation follows noticing that $\nabla \cdot I_{\mathcal{BDM},T}^1 v \in \mathcal{P}^0(T)$ and recalling that $q$ has zero-average on $T$.
  Since $\nabla q$ spans $\mathcal{P}^0(T)^d$ as $q$ spans $\mathcal{P}^1(T) \cap L^2_0(T)$, this concludes the proof of \eqref{eq:pi.I.BDM=pi} for $\ell = 1$.
\end{remark}

\subsubsection{The Raviart--Thomas--N\'ed\'elec space}

For any integer $\ell \ge 0$, let $\rot \coloneq \nabla^\perp$ and define
\[
\mathcal{R}^\ell(T) \coloneq \begin{cases}
  \rot \mathcal{P}^{\ell+1}(T) & \text{if $d=2$}, \\
  \curl \mathcal{P}^{\ell+1}(T)^d & \text{if $d=3$},
\end{cases}\qquad
\mathcal{R}_{\rm c}^\ell(T) \coloneq (x - x_T) \mathcal{P}^{\ell-1}(T).
\]
For any $\ell \ge 1$, the Raviart--Thomas--N\'ed\'elec space is
\[
\mathcal{RTN}^\ell(T)
\coloneq
\mathcal{R}^{\ell-1}(T) \oplus \mathcal{R}_{\rm c}^\ell(T),
\]
furnished with the interpolator $I_{\mathcal{RTN},T}^\ell : H^1(T)^d \to \mathcal{RTN}^\ell(T)$ such that, for all $v \in H^1(T)^d$,
\begin{subequations}\label{eq:I.RTN}
  \begin{align}\label{eq:I.RTN:F}
    \forall F \in \mathcal{F}_T, &\qquad&
    \int_F (I_{\mathcal{RTN},T}^\ell v \cdot n_{TF})\, q &= \int_F (v \cdot n_{TF})\, q
    &\qquad& \forall q \in \mathcal{P}^{\ell-1}(F),
    \\ \label{eq:I.RTN:T}
    && \int_T I_{\mathcal{RTN},T}^\ell v \cdot w &= \int_T v \cdot w
    &\qquad& \forall w \in \mathcal{P}^{\ell-2}(T)^d.    
  \end{align}
\end{subequations}
It is clear from \eqref{eq:I.RTN:T} that
\begin{equation}\label{eq:pi.I.RTN=pi}
  \pi_{\mathcal{P}^{\ell-2}(T)^d} \circ I_{\mathcal{RTN},T}^\ell = \pi_{\mathcal{P}^{\ell-2}(T)^d}.
  \end{equation}
Using the previous relation for $\ell \ge 2$, it is also clear that the Raviart--Thomas--N\'ed\'elec interpolator preserves the average value of the interpolated function on $T$.
This is not the case, however, for $\ell = 1$.

\subsubsection{The rectangular Brezzi--Douglas--Fortin--Marini space}

Let $T$ denote a rectangle (if $d = 2$) or a rectangular parallelepiped (if $d=3$).
For any $\ell \ge 1$, the Brezzi--Douglas--Fortin--Marini of space of \cite{Brezzi.Douglas.ea:87} is 
\[
\mathcal{BDFM}^\ell(T)
\coloneq
\mathcal{P}^{\ell}(T)^d \setminus \mathcal{P}_{\rm hom}^\ell(T)^d,
\]
where, letting $y \coloneq x - x_T$, the space $\mathcal{P}_{\rm hom}^\ell(T)^d$ is defined as
\[
\mathcal{P}_{\rm hom}^\ell(T)^d \coloneq \begin{cases}
\operatorname{span}\left\{ \left( \begin{array}{c}y_2^\ell \\ 0 \end{array}\right),\, \left( \begin{array}{c} 0 \\ y_1^\ell \end{array}\right) \right\} & \text{if $d=2$}, \\
\operatorname{span}\left\{\left( \begin{array}{c}y_2^\ell \\ 0 \\ 0 \end{array}\right),\, \left( \begin{array}{c}y_3^\ell \\ 0 \\ 0 \end{array}\right),\, \left( \begin{array}{c}0 \\ y_1^\ell \\ 0 \end{array}\right),\, \left( \begin{array}{c}0 \\ y_3^\ell \\ 0 \end{array}\right),\, \left( \begin{array}{c}0 \\ 0 \\ y_1^\ell \end{array}\right),\, \left( \begin{array}{c}0 \\ 0\\ y_2^\ell  \end{array}\right) \right\}
& \text{if $d=3$}.
\end{cases}
\]
The space $\mathcal{BDFM}^\ell(T)$ is endowed with the interpolator $I_{\mathcal{BDFM},T}^\ell : H^1(T)^d \to \mathcal{BDFM}^\ell(T)$ such that, for all $v \in H^1(T)^d$,
\begin{subequations}\label{eq:I.BDFM}
  \begin{align}\label{eq:I.BDFM:F}
    \forall F \in \mathcal{F}_T, &\qquad&
    \int_F (I_{\mathcal{BDFM},T}^\ell v \cdot n_{TF})\, q &= \int_F (v \cdot n_{TF})\, q
    &\qquad& \forall q \in \mathcal{P}^{\ell-1}(F),
    \\ \label{eq:I.BDFM:T}
    && \int_T I_{\mathcal{BDFM},T}^\ell v \cdot w &= \int_T v \cdot w
    &\qquad& \forall w \in \mathcal{P}^{\ell-2}(T)^d.    
  \end{align}
\end{subequations}
It is clear from \eqref{eq:I.BDFM:T} that $\pi_{\mathcal{P}^{\ell-2}(T)^d} \circ I_{\mathcal{BDFM},T}^\ell = \pi_{\mathcal{P}^{\ell-2}(T)^d}$.

\subsubsection{Stabilization}

Let a mesh element $T \in \mathcal{T}_h$ be fixed.
In the context of HHO methods, the local stabilization bilinear form $s_T$ in \eqref{eq:ah} hinges on a velocity reconstruction $r_T : \underline{U}_T \to W_T$, with $W_T \subset L^2(T)^d$, defined by
\begin{subequations}\label{eq:rT}
  \begin{gather}\label{eq:rT:pde}
    \int_T \nabla r_T \underline{v}_T : \nabla w
    = - \int_T v_T \cdot \Delta w
    + \sum_{F \in \mathcal{F}_T} \int_F v_F \cdot (\nabla w\, n_{TF})
    \qquad\forall w \in W_T,
    \\ \label{eq:rT:closure}
    \int_T r_T \underline{v}_T = \begin{cases}
      \int_T v_T & \text{if $\pi_{\mathcal{P}^0(T)^d} \circ I_{U,T} = \pi_{\mathcal{P}^0(T)^d}$},
      \\
      \frac{|T|}{\operatorname{card}(\mathcal{F}_T)}\sum_{F \in \mathcal{F}_T} \frac{1}{|F|} \int_F v_F
      & \text{otherwise},
    \end{cases} 
  \end{gather}
\end{subequations}
with $| \cdot |$ denoting the Haussdorff measure.

\begin{remark}[Preservation of the average value of affine functions]\label{rem:rT.avg.value}
  Denote by $x_T \coloneq \frac{1}{|T|} \int_T x$ the center of mass of $T$.
  To check that the second condition in \eqref{eq:rT:closure} ensures that $\int_T r_T \underline{I}_{U,T} v = \int_T v$ for all $v \in \mathcal{P}^1(T)^d$, we write
  \[
  \begin{aligned}
    \int_T r_T \underline{I}_{U,T} v
    = \frac{|T|}{\operatorname{card}(\mathcal{F}_T)} \sum_{F \in \mathcal{F}_T} \frac{1}{|F|} \int_F v
    = \frac{|T|}{\operatorname{card}(\mathcal{F}_T)} \sum_{F \in \mathcal{F}_T} v(x_F)
    = |T|\, v(x_T)
    = \int_T v,
  \end{aligned}
  \]
  where we have used the linearity of $v$ in the second and third equalities and, for the latter, also the fact that $x_T = \frac{1}{\operatorname{card}(\mathcal{F}_T)} \sum_{F \in \mathcal{F}_T} x_F$ with $x_F$ denoting the center of mass of $F \in \mathcal{F}_T$.
\end{remark}

One then obtains $s_T$ penalizing the following quantities:
\begin{equation}\label{eq:delta.T.TF}
  \text{%
    $\delta_T \underline{v}_T \coloneq I_{U,T} ( r_T \underline{v}_T - v_T )$
    and $\delta_{TF} \underline{v}_T \coloneq I_{U,F} (r_T \underline{v}_T - v_F)$ for all $F \in \mathcal{F}_T$.
  }
\end{equation}
The expression for the local stabilization bilinear form $s_T : \underline{U}_T \times \underline{U}_T \to \mathbb{R}$ used in the numerical examples of Section~\ref{sec:numerical.examples} is:
For all $(\underline{w}_T, \underline{v}_T) \in \underline{U}_T \times \underline{U}_T$,
\begin{equation}\label{eq:sT:classical.hho}
   s_T(\underline{w}_T, \underline{v}_T)
   \coloneq
   \lambda_T h_T^{-2}\int_T \delta_T\underline{w}_T \cdot \delta_T\underline{v}_T
   + \sum_{F \in \mathcal{F}_T} h_F^{-1} \int_F \delta_{TF}\underline{w}_T \cdot \delta_{TF}\underline{v}_T,
\end{equation}
where the purpose of $\lambda_T \coloneq \operatorname{card}(\mathcal{F}_T) \frac{h_T^d}{|T|} \simeq 1$ is to equilibrate the volumetric and boundary contributions; cf. \cite{Di-Pietro.Droniou:23*2}.
Many other choices for $s_T$ are possible, the only requirements being that the dependence on the arguments is via the difference operators $\delta_T$ and $\delta_{TF}$, $F \in \mathcal{F}_T$, and that the scaling in $h_T$ is the same as that of \eqref{eq:sT:classical.hho}.
We refer to \cite[Chapter~2]{Di-Pietro.Droniou:20} for a more general discussion highlighting the key properties of HHO stabilizing bilinear forms and further examples.

\subsection{The classical Botti--Massa method}\label{sec:applications:botti-massa}

We assume here that $\mathcal{M}_h$ is a standard matching simplicial mesh that belongs to a regular sequence.
The classical Botti--Massa method of \cite{Botti.Massa:22} corresponds to the following choice of spaces for a given polynomial degree $k \ge 0$:
\begin{equation}\label{eq:botti-massa:spaces}
  U_T = \mathcal{BDM}^{k+1}(T),\quad
  U_F = \mathcal{P}^k(F)^d,\quad
  P_T = \mathcal{P}^k(T),\quad
  P_F = \mathcal{P}^{k+1}(F),
\end{equation}
along with the local stabilization \eqref{eq:sT:classical.hho} obtained reconstructing the velocity in $W_T = \mathcal{P}^{k+1}(T)^d$.
The space $\Sigma_T$ can be taken equal to $\nabla \mathcal{P}^{k+1}(T)^d$ or $\mathcal{P}^k(T)^{d\times d}$.
Notice that the expression of $\delta_T$ is actually simpler in this case: observing that $r_T \underline{v}_T - v_T \in \mathcal{P}^{k+1}(T)^d$ for all $\underline{v}_T \in \underline{U}_T$ and that $I_{U,T} = I_{\mathcal{BDM},T}^{k+1}$ restricted to $\mathcal{P}^{k+1}(T)^d$ is the identity, we can simply write $\delta_T \underline{v}_T = r_T \underline{v}_T - v_T$.
This stabilization classically satisfies Assumption~\ref{ass:sT} as well as the other standard properties of HHO stabilizations; see \cite[Chapter~2]{Di-Pietro.Droniou:20}.
For the polynomial consistency property, in particular, we use the fact that $r_T \circ \underline{I}_{U,T}$ coincides (component-wise) with the elliptic projector of \cite[Definition~1.39]{Di-Pietro.Droniou:20}.

\begin{lemma}[$\|\cdot\|_{\nu,h}$-norm]  
  The map $\|\cdot\|_{\nu,h}$ defined by \eqref{eq:norm.h} is a norm on $\underline{U}_{h,0} \times \underline{P}_{h,0}$ with local spaces given by \eqref{eq:botti-massa:spaces}.
\end{lemma}

\begin{proof}
  It suffices to show that, for all $(\underline{v}_h, \underline{q}_h) \in \underline{U}_{h,0} \times \underline{P}_{h,0}$, $\| (\underline{v}_h, \underline{q}_h) \|_{\nu,h} = 0$ implies $(\underline{v}_h, \underline{q}_h) = (\underline{0}, \underline{0})$.
  We start by noticing that $\| \underline{v}_h \|_{1,h} = 0$ classically implies $\underline{v}_h = \underline{0}$ (see, e.g., \cite[Corollary~2.16]{Di-Pietro.Droniou:20}) and that $\| \underline{q}_h \|_{{\rm P},h} = 0$ implies $q_T = 0$ and $G_T \underline{q}_T = 0$ for all $T \in \mathcal{T}_h$.
  It only remains to prove that $q_F = 0$ for all $F \in \mathcal{F}_T$.
  To this purpose, we notice that, for all $w \in \mathcal{BDM}^{k+1}(T)$, accounting for the fact that $v_T = 0$,
  \begin{equation}\label{eq:botti-massa:norm.1.h}
    0 = \int_T G_T \underline{q}_T \cdot w
    \overset{\eqref{eq:GT}}= \sum_{F \in \mathcal{F}_T} \int_F q_F\, (w \cdot n_{TF}).
  \end{equation}
  Taking $w$ such that $w \cdot n_F = q_F$ for all $F \in \mathcal{F}_T$ (possible by virtue of \eqref{eq:I.BDM}) concludes the proof.
\end{proof}

\begin{theorem}[Properties of the Botti--Massa method]\label{thm:botti-massa}
  Let $\mathcal{M}_h$ denote a conforming simplicial mesh.
  Then, the method of \cite{Botti.Massa:22} verifies Assumptions \ref{ass:PT.vs.UT}--\ref{ass:sT}.
  Moreover, assuming that the weak solution $(u,p) \in H_0^1(\Omega)^d \times L_0^2(\Omega)$ of the Stokes problem \eqref{eq:stokes} satisfies the additional regularity $u \in H^{r+2}(\mathcal{T}_h)^d$ for some $r \in \{0, \ldots, k\}$ and $p \in H^1(\Omega)$, it holds, for the error components defined by \eqref{eq:error.components},
  \[
  \text{%
    $\mathcal{E}_u \lesssim h^{r+1} | u |_{H^{r+2}(\mathcal{T}_h)^d}$
    and $\mathcal{E}_p = 0$.
  }
  \]
\end{theorem}

\begin{proof}
  Let us consider Assumption~\ref{ass:PT.vs.UT} for a generic mesh element $T \in \mathcal{T}_h$.
  To check the first point, we start by noticing that $\nabla P_T \subset \mathcal{P}^{k-1}(T)^d \subset \mathcal{BDM}^{k+1}(T)$ and then apply $\pi_{\nabla P_T}$ to \eqref{eq:pi.I.BDM=pi} with $\ell = k+1$ to write $\pi_{\nabla P_T} \circ \pi_{\mathcal{P}^{k-1}(T)^d} \circ I_{U,T} = \pi_{\nabla P_T} \circ \pi_{\mathcal{P}^{k-1}(T)^d}$, notice that $\pi_{\nabla P_T} \circ \pi_{\mathcal{P}^{k-1}(T)^d} = \pi_{\nabla P_T}$ since $\nabla P_T$ is a subspace of $\mathcal{P}^{k-1}(T)^d$, and recall that the composition is associative.
  The second point in Assumption~\ref{ass:PT.vs.UT} is straightforward noticing that $\tr_F P_T = \mathcal{P}^k(F) = U_F \cdot n_{TF}$.
  
  Assumption~\ref{ass:UT.vs.PT} is trivial to check since we have, for all $T \in \mathcal{T}_h$, $\nabla \cdot \mathcal{BDM}^{k+1}(T) = \mathcal{P}^k(T) = P_T$ and $\mathcal{BDM}^{k+1}(T) \cdot n_{TF} = \mathcal{P}^{k+1}(F) = P_F$ for all $F \in \mathcal{F}_T$.

  Let $v \in H_0^1(\Omega)^d$.
  To check Assumption~\ref{ass:global.ibp}, we notice that, by \eqref{eq:I.BDM:F} with $\ell = k+1$, $I_{\mathcal{BDM},T}^{k+1} v \cdot n_{TF} = \pi_{\mathcal{P}^{k+1}(F)} (v \cdot n_{TF})$ for all $T \in \mathcal{T}_h$ and all $F \in \mathcal{F}_T$ and write
  \[
  \begin{aligned}
    \mathcal{E}_{{\rm ibp},h}(v,\underline{q}_h)
    &= \sum_{T \in \mathcal{T}_h} \sum_{F \in \mathcal{F}_T} \int_F \left( \cancel{\pi_{\mathcal{P}^{k+1}(F)}} (v \cdot n_{TF}) - \pi_{\mathcal{P}^k(F)^d} v \cdot n_{TF} \right)\,(q_F - q_T)
    \\
    & =
    \sum_{T \in \mathcal{T}_h} \sum_{F \in \mathcal{F}_T} \int_F ( v  - \pi_{\mathcal{P}^k(F)^d} v ) \cdot n_{TF} \,q_F
    - \sum_{T \in \mathcal{T}_h} \sum_{F \in \mathcal{F}_T} \int_F \left[ v \cdot n_{TF} - \pi_{\mathcal{P}^k(F)} (v \cdot n_{TF}) \right] q_T
    \\
    & =
    \sum_{F \in \mathcal{F}_h^{\rm i}} \int_F \llbracket ( v  - \pi_{\mathcal{P}^k(F)^d} v )\,q_F \rrbracket_F \cdot n_F
    + \sum_{F \in \mathcal{F}_h^{\rm b}} \int_F ( v  - \pi_{\mathcal{P}^k(F)^d} v )\cdot n_F\,q_F 
    \\
    &\quad
    - \sum_{T \in \mathcal{T}_h} \sum_{F \in \mathcal{F}_T} \int_F \left[ v \cdot n_{TF} - \pi_{\mathcal{P}^k(F)} (v \cdot n_{TF}) \right]\,q_T
    = 0,
    \end{aligned}
  \]
  where the cancellation of the projector in the first line is made possible by the fact that $q_F - q_T \in \mathcal{P}^{k+1}(F)$,
  while, in the third equality, we have denoted by $\llbracket \cdot \rrbracket_F$ the jump operator defined consistently with $n_F$ and used the fact that $\pi_{\mathcal{P}^k(F)^d}v \cdot n_{TF} = \pi_{\mathcal{P}^k(F)}(v \cdot n_{TF})$ since faces are planar (hence $n_{TF}$ is constant on $F$).
  The conclusion follows noticing that the jump vanishes since the quantity inside it is single-valued at interfaces for the first term,
  the fact that both $v$ and $\pi_{\mathcal{P}^k(F)^d} v$ vanish on boundary faces for the second term,
  and the fact that, by definition of $\pi_{\mathcal{P}^k(F)}$, $\left[ v \cdot n_{TF} - \pi_{\mathcal{P}^k(F)} (v \cdot n_{TF}) \right]$ is $L^2(F)$-orthogonal to $q_{T|F} \in \mathcal{P}^k(F)$ for the last term.
  
  The first point in Assumption~\ref{ass:SigmaT} follows noticing that $\nabla \cdot \Sigma_T \subset \mathcal{P}^{k-1}(T)^d \subset U_T$ and invoking \eqref{eq:pi.I.BDM=pi} with $\ell = k + 1$.
  The second point is straightforward observing that $\Sigma_T n_{TF} = \mathcal{P}^k(F)^d = U_F$.
  
  Finally, we have already noticed that Assumption~\ref{ass:sT} is verified.
  
  Let us now move to the estimates of the error components.
  The fact that $\mathcal{E}_p = 0$ is a consequence of Remark \ref{rem:pressure.robustness}.
  Concerning $\mathcal{E}_u$, we have
  \begin{equation}\label{eq:Eu:decomposition}
    \mathcal{E}_u^2
    = \sum_{T \in \mathcal{T}_h}
    h_T \| \nabla u - \pi_{\Sigma_T} (\nabla u) \|_{L^2(\partial T)^{d \times d}}^2
    + \sum_{T \in \mathcal{T}_h} s_T(\underline{I}_{U,T} u, \underline{I}_{U,T} u)
    \eqcolon \mathfrak{T}_1 + \mathfrak{T}_2.
  \end{equation}
  By the approximation properties of $\pi_{\Sigma_T} \in \{ \pi_{\nabla \mathcal{P}^{k+1}(T)^d}, \pi_{\mathcal{P}^k(T)^{d \times d}} \}$, $\mathfrak{T}_1 \lesssim h^{r+1} | u |_{H^{r+2}(\mathcal{T}_h)^d}$.
  To estimate the second term as $\mathfrak{T}_2 \lesssim h^{r+1} | u |_{H^{r+2}(\mathcal{T}_h)^d}$, we can reproduce verbatim the argument of \cite[Proposition~2.14]{Di-Pietro.Droniou:20}, which is not affected by the different choice of $U_T$ and $I_{U,T}$.
\end{proof}

\subsection{The Rhebergen--Wells method}\label{sec:rhebergen-wells}

Let $\mathcal{M}_h$ be as in the previous section and an integer $k \ge 1$ be fixed.
The Rhebergen--Wells method of \cite{Rhebergen.Wells:18} applied to the Stokes problem is based on the following spaces:
\[
U_T = \mathcal{BDM}^k(T)^d,\quad
U_F = \mathcal{P}^k(F)^d,\quad
P_T = \mathcal{P}^{k-1}(T),\quad
P_F = \mathcal{P}^k(F),
\]
and we let $I_{U,T} = I_{\mathcal{BDM},T}^k$ for all $T \in \mathcal{T}_h$.
The choice for the space $\Sigma_T$ is not explicit in the original paper, but we can take $\Sigma_T = \mathcal{P}^{k-1}(T)^{d\times d}$ or $\Sigma_T = \nabla \mathcal{P}^k(T)^d$ to fit the method into the abstract framework developed in the previous sections.
For any $T \in \mathcal{T}_h$, the original form of the local viscous bilinear form reads:
For all $(\underline{w}_T, \underline{v}_T) \in \underline{U}_T \times \underline{U}_T$,
\begin{equation}\label{eq:RW:aT}
  \begin{aligned}
    a_T(\underline{w}_T, \underline{v}_T)
    &= \int_T \nabla w_T : \nabla v_T
    + \sum_{F \in \mathcal{F}_T} \int_F (w_F - w_T) \cdot (\nabla v_T n_{TF})  
    + \sum_{F \in \mathcal{F}_T} \int_F (\nabla w_T n_{TF}) \cdot (v_F - v_T)
    \\
    &\quad
    + \eta \sum_{F \in \mathcal{F}_T} h_F^{-1} \int_F (w_F - w_T) \cdot (v_F - v_T),
  \end{aligned}
\end{equation}
where $\eta > 0$ is a stabilization parameter to be chosen large enough.
Noticing that both $\nabla w_T$ and $\nabla v_T$ belong to $\mathcal{P}^{k-1}(T)^{d \times d}$, we can use the definition \eqref{eq:ET'} of $E_T$ with $(\underline{v}_T, \tau)$ equal to $(\underline{w}_T, \nabla v_T)$ and $(\underline{v}_T, \nabla w_T)$ to express the second and third terms respectively as $\int_T (E_T \underline{w}_T - \nabla w_T): \nabla v_T$ and $\int_T \nabla w_T : (E_T \underline{v}_T - \nabla v_T)$.
Substituting into \eqref{eq:RW:aT} and rearranging, we get
\[
  a_T(\underline{w}_T, \underline{v}_T)
  = \int_T E_T \underline{w}_T : E_T \underline{v}_T
  - \int_T (E_T \underline{w}_T - \nabla w_T) : (E_T \underline{v}_T - \nabla v_T)
  + \eta \sum_{F \in \mathcal{F}_T} h_F^{-1} \int_F (w_F - w_T) \cdot (v_F - v_T).
\]
This brings us to the expression \eqref{eq:ah} with
\begin{equation}\label{eq:RW:sT}
  s_T(\underline{w}_T, \underline{v}_T)
  = - \int_T (E_T \underline{w}_T - \nabla w_T) : (E_T \underline{v}_T - \nabla v_T)
  + \eta \sum_{F \in \mathcal{F}_T} h_F^{-1} \int_F (w_F - w_T) \cdot (v_F - v_T).
\end{equation}

The results stated in the following theorem are consistent with the analysis carried out in \cite{Kirk.Rhebergen:19}.

\begin{theorem}[Properties of the Rhebergen--Wells method]\label{thm:RW}
  Let $\mathcal{M}_h$ denote a conforming simplicial mesh.  
  Then, assuming that $\eta > (d + 1) C_{\rm tr}^2$ with $C_{\rm tr}$ discrete trace inequality constant (cf., e.g., \cite[Lemma~1.46]{Di-Pietro.Ern:12}), the Rhebergen--Wells method verifies Assumptions \ref{ass:PT.vs.UT}--\ref{ass:sT}.
  Moreover, assuming that $(u,p) \in H_0^1(\Omega)^d \times L_0^2(\Omega)$ solving the Stokes problem \eqref{eq:stokes} satisfies the additional regularity $u \in H^{r+1}(\mathcal{T}_h)^d$ for some $r \in \{0, \ldots, k\}$ and $p \in H^1(\Omega)$, it holds, for the error components defined by \eqref{eq:error.components},
  \[
  \text{%
    $\mathcal{E}_u \lesssim h^r | u |_{H^{r+1}(\mathcal{T}_h)^d}$
    and $\mathcal{E}_p = 0$.
  }
  \]
\end{theorem}

\begin{proof}
  The proof of Assumptions~\ref{ass:PT.vs.UT} and~\ref{ass:UT.vs.PT} is the same as in Theorem~\ref{thm:botti-massa} provided we replace $k$ with $k - 1$ (the only difference being that, this time, $\tr_F P_T = \mathcal{P}^{k-1}(F)$ is strictly contained in $U_F \cdot n_{TF}$ instead of coinciding with this space).

  To check Assumption~\ref{ass:global.ibp}, it suffices to notice that, for all $v \in H^1_0(\Omega)^d$, all $T \in \mathcal{T}_h$, and all $F \in \mathcal{F}_T$, $(I_U v)_{|F} \cdot n_{TF} - I_F v_{|F} \cdot n_{TF} \overset{\eqref{eq:I.BDM:F}}= \pi_{\mathcal{P}^k(F)} (v \cdot n_{TF}) - \pi_{\mathcal{P}^k(F)^d} v \cdot n_{TF} = 0$, where the last passage uses the fact that $n_{TF}$ is constant on $F$.

  The first point in Assumption~\ref{ass:SigmaT} is proved as in Theorem~\ref{thm:botti-massa} replacing $k$ with $k - 1$ while, for the second point, we write $\Sigma_T n_{TF} \subset \mathcal{P}^{k-1}(F)^d \subset U_F$.

  Finally, Assumption~\ref{ass:sT} can be proved for $\eta > (d + 1) C_{\rm tr}^2$ starting from the form \eqref{eq:RW:aT} and using standard techniques in the context of Discontinuous Galerkin methods; cf., e.g., \cite[Lemma~4.12]{Di-Pietro.Ern:12}.

  Let us now move to the estimates of the error components.
  The fact that $\mathcal{E}_p = 0$ is a consequence of Remark \ref{rem:pressure.robustness}.
  Concerning $\mathcal{E}_u$, we start from a decomposition similar to \eqref{eq:Eu:decomposition}.
  For the first term, we readily obtain the estimate $\mathfrak{T}_1 \lesssim h^r | u |_{H^{r+1}(\mathcal{T}_h)^d}$.
  To estimate the term involving the stabilization bilinear form, we write
  \[
  \begin{aligned}
    s_T(\underline{I}_{U,T} u, \underline{I}_{U,T} u)
    \overset{\eqref{eq:RW:sT}}&\le \| E_T \underline{I}_{U,T} u - \nabla I_{\mathcal{BDM},T}^k u \|_{L^2(T)^{d\times d}}^2
    + \eta \sum_{F \in \mathcal{F}_T} h_F^{-1} \| \pi_{\mathcal{P}^k(F)^d} u - I_{\mathcal{BDM},T}^k u \|_{L^2(F)^d}^2
    \\
    &\lesssim
    \| \pi_{\Sigma_T} \nabla u - \nabla u \|_{L^2(T)^{d\times d}}^2
    + \| \nabla (u -  I_{\mathcal{BDM},T}^k u ) \|_{L^2(T)^{d\times d}}^2
    \\
    &\quad
    + \sum_{F \in \mathcal{F}_T} h_F^{-1} \| u - I_{\mathcal{BDM},T}^k u \|_{L^2(F)^d}^2
    \lesssim h_T^r | u |_{H^{r+1}(T)^d},
  \end{aligned}
  \]
  where, to pass to the second line, we have inserted $\pm \nabla u$ and used a triangle inequality together with the fact that $E_T \underline{I}_{U,T} u = \pi_{\Sigma_T} \nabla u$ for the first term,
  and noticed that, since the trace on $F$ of $I_{\mathcal{BDM},T}^k u$ is in $\mathcal{P}^k(F)^d$,
  $ {\| \pi_{\mathcal{P}^k(F)^d} u - I_{\mathcal{BDM},T}^k u \|_{L^2(F)^d}} =  {\| \pi_{\mathcal{P}^k(F)^d} ( u - I_{\mathcal{BDM},T}^k u ) \|_{L^2(F)^d}} \le  {\| u - I_{\mathcal{BDM},T}^k u \|_{L^2(F)^d}}$.
  The fact that $\mathfrak{T}_2 \lesssim h^r | u |_{H^{r+1}(\mathcal{T}_h)^d}$ is then a consequence of standard approximation results for $\pi_{\Sigma_T} \in \{ \pi_{\nabla \mathcal{P}^k(T)^d}, \pi_{\mathcal{P}^{k-1}(T)^{d\times d}} \}$ (see, e.g., \cite[Theorem~1.45]{Di-Pietro.Droniou:20}) and $I_{\mathcal{BDM},T}^k$ (see \cite[Proposition~2.5.1]{Boffi.Brezzi.ea:13}).
\end{proof}

\subsection{A new method with Raviart--Thomas--N\'ed\'elec velocities at elements}\label{sec:applications:rtn-pk-pk-pk}

We assume again that $\mathcal{M}_h$ is a standard matching simplicial mesh belonging to a regular sequence and consider a new method based on the following component spaces for a given polynomial degree $k \ge 0$:
\[
  U_T = \mathcal{RTN}^{k+1}(T),\quad
  U_F = \mathcal{P}^k(F)^d,\quad
  P_T = \mathcal{P}^k(T),\quad
  P_F = \mathcal{P}^k(F),
\]
with $I_{U,T} = I_{\mathcal{RTN},T}^{k+1}$.
The space $\Sigma_T$ can be taken equal to $\nabla \mathcal{P}^{k+1}(T)^d$ or $\mathcal{P}^k(T)^{d\times d}$.
The stabilization is again given by \eqref{eq:sT:classical.hho} with velocity reconstruction in $W_T = \mathcal{P}^{k+1}(T)^d$.
Notice that, with this choice, we have, by \eqref{eq:delta.T.TF}, $\delta_T \underline{v}_T = I_{\mathcal{RTN},T}^{k+1} (r_T \underline{v}_T - v_T)$ for all $\underline{v}_T \in \underline{U}_T$, i.e., the Raviart--Thomas--N\'ed\'elec interpolator replaces the $L^2$-orthogonal projection usually present in HHO methods.

The fact that $s_T$ satisfies the standard assumptions of \cite[Chapter~2]{Di-Pietro.Droniou:20} on HHO stabilizations can be proved using standard techniques in the context of HHO methods.
For the polynomial consistency property, in particular, we use the fact that, recalling Remark \ref{rem:rT.avg.value}, $r_T \circ \underline{I}_{U,T}$ is a modified elliptic projector in the spirit of \cite[Definition~5.4]{Di-Pietro.Droniou:20}.

The fact that the map defined by \eqref{eq:norm.h} defines a norm on $\underline{U}_{h,0} \times \underline{P}_{h,0}$ can be proved in the same way as for the Botti--Massa method, noticing that it is still possible to select in \eqref{eq:botti-massa:norm.1.h} $v \in \mathcal{RTN}^{k+1}(T)$ such that $v \cdot n_F = q_F \in \mathcal{P}^k(F)$.

\begin{theorem}[Properties of the new method with Raviart--Thomas--N\'ed\'elec element velocity]\label{thm:rtn}
  Let $\mathcal{M}_h$ denote a conforming simplicial mesh.  
  Then, the method described in this section verifies Assumptions \ref{ass:PT.vs.UT}--\ref{ass:sT}.
  Moreover, assuming that the solution $(u,p) \in H_0^1(\Omega)^d \times L_0^2(\Omega)$ of the Stokes problem \eqref{eq:stokes} satisfies the additional regularity $u \in H^{r+2}(\mathcal{T}_h)^d$ for some $r \in \{0, \ldots, k\}$ and $p \in H^1(\Omega)$, it holds, for the error components defined by \eqref{eq:error.components},
  \[
  \text{%
    $\mathcal{E}_u \lesssim h^{r+1} | u |_{H^{r+2}(\mathcal{T}_h)^d}$
    and $\mathcal{E}_p = 0$.
  }
  \]
\end{theorem}

\begin{remark}[Comparison with the Rhebergen--Wells method]
  After static condensation of the element unknowns (with the possible exception of one pressure unknown per element), the method presented in this section yields linear systems with analogous size and pattern as the method of Rhebergen--Wells of Section~\ref{sec:rhebergen-wells}.
  However, the present method converges in $h^{k+1}$ as opposed to $h^k$.
  This results from two important differences: first, the fact that we use a slightly larger space for the element velocity ($\mathcal{RTN}^{k+1}(T)$ as opposed to $\mathcal{P}^k(T)^d$); second, the fact that we use an HHO-type viscous stabilization, cf. \cite{Cockburn.Di-Pietro.ea:16} on this subject.
\end{remark}

\begin{proof}[Proof of Theorem~\ref{thm:rtn}]
  The proof Assumption~\ref{ass:PT.vs.UT} for a given mesh element $T \in \mathcal{T}_h$ is obtained repeating the proof of the corresponding point in Theorem~\ref{thm:botti-massa} with the following substitutions: $\mathcal{BDM}^{k+1}(T) \gets \mathcal{RTN}^{k+1}(T)$ and
  $\eqref{eq:pi.I.BDM=pi} \gets \eqref{eq:pi.I.RTN=pi}$.
  Assumption~\ref{ass:UT.vs.PT} follows noticing that, for all $T \in \mathcal{T}_h$, $\nabla \cdot \mathcal{RTN}^{k+1}(T) = \mathcal{P}^k(T) = P_T$ and, for all $F \in \mathcal{F}_T$, $\mathcal{RTN}^{k+1}(T) \cdot n_F = \mathcal{P}^k(F) = P_F$.
  Assumption~\ref{ass:global.ibp} immediately follows noticing that, for all $v \in H^1_0(\Omega)^d$, all all $T \in \mathcal{T}_h$, and all $F \in \mathcal{F}_T$, $I_{U,T} v \cdot n_{TF} \overset{\eqref{eq:I.RTN:F}}= \pi_{\mathcal{P}^k(F)} (v \cdot n_{TF}) = \pi_{\mathcal{P}^k(F)^d}v \cdot n_{TF} = I_{U,F} v \cdot n_{TF}$, where the second equality holds since $n_{TF}$ is constant over $F$.
  The proof of Assumption~\ref{ass:SigmaT} is identical to the one given in Theorem \ref{thm:botti-massa} provided \eqref{eq:pi.I.BDM=pi} is replaced by \eqref{eq:pi.I.RTN=pi}.
  The fact that Assumption~\ref{ass:sT} holds has already been observed right after the definition of $s_T$.
  Finally, the estimates of the error components are essentially identical to the ones given in Theorem \ref{thm:botti-massa}, the only difference being that the approximation properties of the modified elliptic projector stated in \cite[Theorem~5.7]{Di-Pietro.Droniou:20} have to be invoked when $k = 0$ to estimate the term involving $s_T$.
  This concludes the proof.
\end{proof}

\begin{remark}[Brezzi--Douglas--Fortin--Marini velocities at elements for rectangular meshes]\label{rem:bdfm_mixedmeshes}
When $\mathcal{M}_h$ is a matching rectangular mesh, the previous analysis applies without modifications to the space choice
  \[
  U_T = \mathcal{BDFM}^{k+1}(T),\quad
  U_F = \mathcal{P}^k(F)^d,\quad
  P_T = \mathcal{P}^k(T),\quad
  P_F = \mathcal{P}^k(F),
  \]
  with polynomial degree $k \ge 0$ and $I_{U,T} = I_{\mathcal{BDFM},T}^{k+1}$ defined in \eqref{eq:I.BDFM}, under the assumption that $\mathcal{M}_h$ is a rectangular mesh belonging to a regular sequence. The velocity reconstruction for defining the stabilization \eqref{eq:sT:classical.hho} is again taken in $W_T = \mathcal{P}^{k+1}(T)^d$, while the velocity gradient reconstruction can be taken in $\nabla \mathcal{P}^{k+1}(T)^d$ or $\mathcal{P}^k(T)^{d\times d}$.
\end{remark}

\begin{remark}[A variant with full polynomial velocities at elements]\label{rem:pkpo-pk-pk-pk}
  A variant of the method with $U_T = \mathcal{P}^{k+1}(T)^d$ (still keeping $I_{U,T} = I_{\mathcal{RTN},T}^{k+1}$ if $T$ is a simplex and $I_{U,T} = I_{\mathcal{BDFM},T}^{k+1}$ if $T$ is rectangular) is obtained using the following viscous stabilisation:
  \begin{multline}\label{eq:rtn:Pk+1}
    s_T(\underline{w}_T, \underline{v}_T)
    \coloneq
    \lambda_T h_T^{-2}\int_T \delta_T\underline{w}_T \cdot \delta_T\underline{v}_T
    + \sum_{F \in \mathcal{F}_T} h_F^{-1} \int_F \delta_{TF}\underline{w}_T \cdot \delta_{TF}\underline{v}_T    
    \\
    + \boxed{
      \nu^{-1} \lambda_T h_T^{-2} \int_T (w_T - I_{U,T} w_T) \cdot (v_T - I_{U,T} v_T).
    }
  \end{multline}
  The proof that Assumptions \ref{ass:PT.vs.UT}--\ref{ass:sT} are verified is analogous to the one in Theorem~\ref{thm:rtn}.
  In particular, the fact that, for all $w \in H^{r+2}(T)^d$ with $r \in \{0,\ldots k\}$,
  $
  s_T(\underline{I}_{U,T} w, \underline{v}_T)
  \lesssim h^{r+1} | w |_{H^{r+2}(T)^d} \| \underline{v}_T \|_{1,T}
  $
  follows noticing that the boxed term in \eqref{eq:rtn:Pk+1} vanishes when $\underline{w}_T = \underline{I}_{U,T} w$.
\end{remark}

\subsection{A new method on general polytopal meshes}\label{sec:applications:polytopal}

To close this section, we consider a method on general polytopal meshes with spaces
\[ 
  U_T = \mathcal{BDM}^{k+1}(T),\quad
  U_F = \mathcal{P}^k(F)^d,\quad
  P_T = \mathcal{P}^k(T),\quad
  P_F = \mathcal{P}^k(F),
\] 
and $I_{U,T} \coloneq \pi_{\mathcal{P}^{k+1}(T)^d}$.
For $\Sigma_T$, we can take $\nabla \mathcal{P}^{k+1}(T)^d$ or $\mathcal{P}^k(T)^{d\times d}$.
The stabilization is again \eqref{eq:sT:classical.hho}, with velocity reconstruction in $W_T = \mathcal{P}^{k+1}(T)^d$.
This corresponds to the standard HHO discretization of viscous terms.

\begin{theorem}[Properties of the polytopal method]
  Let $\mathcal{M}_h$ denote a general polytopal mesh belonging to a regular refined sequence, in the sense made precise in \cite[Chapter~1]{Di-Pietro.Droniou:20}.
  Then, the method described in this section verifies Assumptions \ref{ass:PT.vs.UT}, \ref{ass:UT.vs.PT}, \ref{ass:SigmaT}, and \ref{ass:sT}.
  Moreover, assuming that the weak solution $(u,p) \in H_0^1(\Omega)^d \times L_0^2(\Omega)$ of the Stokes problem \eqref{eq:stokes} satisfies the additional regularity $u \in H^{r+2}(\mathcal{T}_h)^d$ and $p \in H^1(\Omega) \cap H^{r+1}(\mathcal{T}_h)$ for some $r \in \{0, \ldots, k\}$, it holds, for the error components defined by \eqref{eq:error.components},
  \begin{equation}\label{eq:error.estimates:polytopal}
    \text{%
      $\mathcal{E}_u \lesssim h^{r+1} | u |_{H^{r+2}(\mathcal{T}_h)^d}$
      and $\mathcal{E}_p \lesssim h^{r+1} | p |_{H^{r+1}(\mathcal{T}_h)}$.
    }
  \end{equation}
\end{theorem}

\begin{proof}
  The proof of Assumptions \ref{ass:PT.vs.UT}, \ref{ass:UT.vs.PT}, and \ref{ass:SigmaT} is trivial given the choice of spaces and the fact that $I_{U,T}$ is taken equal to the $L^2$-orthogonal projector on $\mathcal{P}^{k+1}(T)^d$.
  Assumption \ref{ass:global.ibp} is not verified in general, which requires a pressure stabilization term in the scheme.
  Assumption \ref{ass:sT} is proved in \cite[Chapter~5]{Di-Pietro.Droniou:20}.

  Let us prove the estimate $\mathcal{E}_u$.
  The fact that the first two contributions in the right-hand side of \eqref{eq:Eu} are $\lesssim | u |_{H^{r+2}(\mathcal{T}_h)^d}$ is completely standard in the context of HHO methods.
  Let us focus on the last contribution.
  Let $\underline{q}_h \in \underline{P}_h$.
  Carrying out the appropriate substitutions in \eqref{eq:global.ibp}, we have
  \[
  \begin{aligned}
    \mathcal{E}_{{\rm ibp},h}(u, \underline{q}_h)
    &= \sum_{T \in \mathcal{T}_h} \sum_{F \in \mathcal{F}_T} \int_F (\pi_{\mathcal{P}^{k+1}(T)^d} u - \pi_{\mathcal{P}^k(F)^d} u) \cdot n_{TF} (q_F - q_T)
    \\
    &= \sum_{T \in \mathcal{T}_h} \sum_{F \in \mathcal{F}_T} \int_F \pi_{\mathcal{P}^k(F)^d}(\pi_{\mathcal{P}^{k+1}(T)^d} u -  \cancel{\pi_{\mathcal{P}^k(F)^d}} u) \cdot n_{TF} (q_F - q_T)
    \\
    &\le \sum_{T \in \mathcal{T}} \sum_{F \in \mathcal{F}_T} h_F^{-\frac12} \| \pi_{\mathcal{P}^{k+1}(T)^d} u -  u \|_{L^2(F)^d}\,h_F^{\frac12} \| q_F - q_T \|_{L^2(F)}
    \\
    &\lesssim h^{r+1} | u |_{H^{r+2}(\mathcal{T}_h)^d}\, | \underline{q}_h |_{0,h},
  \end{aligned}
  \]
  where we have used the fact that $q_F - q_{T|F} \in \mathcal{P}^k(F)$ to insert $\pi_{\mathcal{P}^k(F)^d}$ in front of the first parentheses and its linearity and idempotency to cancel it inside them in the second line,
  $(2,\infty,2)$-H\"{o}lder inequalities followed by $\| n_{TF} \|_{L^\infty(F)^d} \le 1$ and the $L^2$-continuity of $\pi_{\mathcal{P}^k(F)^d}$ in the third line,
  and the approximation properties of $\pi_{\mathcal{P}^{k+1}(T)^d}$ together with the definition \eqref{eq:seminorm.0h} of $| \cdot |_{0,h}$ to conclude.
  Thus,
  $\sup_{\underline{q}_h \in \underline{P}_h,\, | \underline{q}_h |_{0,h} = 1} \mathcal{E}_{{\rm ibp},h}(u,\underline{q}_h)
  \lesssim h^{r+1} | u |_{H^{r+2}(\mathcal{T}_h)^d}$,
  which concludes the proof of the estimate of $\mathcal{E}_u$ in \eqref{eq:error.estimates:polytopal}.
  To estimate $\mathcal{E}_p$, it suffices to use the linearity and idempotency of $\pi_{\mathcal{P}^k(F)}$ followed by its $L^2$-continuity as above to write
  \[
  | \underline{I}_{P,h} p |_{0,h}^2
  = \sum_{T \in \mathcal{T}_h} h_T \sum_{F \in \mathcal{F}_T} \| \pi_{\mathcal{P}^k(F)} (p - \pi_{\mathcal{P}^k(T)} p )\|_{L^2(F)}^2
  \le \sum_{T \in \mathcal{T}_h} h_T \sum_{F \in \mathcal{F}_T} \| p - \pi_{\mathcal{P}^k(T)} p \|_{L^2(F)}^2
  \]
  and conclude using the approximation properties of $\pi_{\mathcal{P}^k(T)}$.
  This gives $\mathcal{E}_p \lesssim h^{r+1} | p |_{H^{r+1}(\mathcal{T}_h)}$ and concludes the proof.
\end{proof}

\section{Numerical examples}\label{sec:numerical.examples}

We provide here a numerical validation of the properties of some of the methods described in Section~\ref{sec:applications}.
For $\Omega = (0,1)^2$, we consider the following manufactured analytical solution, originally proposed in \cite{Lederer.Linke.ea:17}:
\[
u(x,y) = \begin{pmatrix}
  x^2 (x - 1)^2 (4 y^3 - 6 y^2 + 2y) \\
  -y^2 (y - 1)^2 (4 x^3 - 6 x^2 + 2 x)
\end{pmatrix},\qquad
p(x,y) = x^7 + y^7 - \frac14.
\]
The momentum forcing term $f$ is inferred from the previous expressions and homogeneous Dirichlet boundary conditions are enforced.
The average value of the pressure is enforced by means of the Lagrange multipliers method.
For the Botti--Massa method of Section~\ref{sec:applications:botti-massa} and the new method with Raviart--Thomas--Nédélec velocities at elements of Section~\ref{sec:applications:rtn-pk-pk-pk}, we solve this problem on a sequence of triangular meshes obtained by uniform refinement of the one depicted in Figure~\ref{fig:meshes:tri}.
For the polygonal method of Section~\ref{sec:applications:polytopal}, we additionally consider the Cartesian, locally refined, and hexagonal meshes in Figure~\ref{fig:meshes:cart}--\ref{fig:meshes:hexa}.
We monitor the discrete error components $\| \underline{u}_h - \underline{I}_{U,h} u \|_{1,h}$,
$\| u_h - I_{U,h} u \|_{L^2(\Omega)^2}$ (with $(I_{U,h} u)_{|T} \coloneq I_{U,T} u$ for all $T \in \mathcal{T}_h$),
and $\| p_h - I_{P,h} p \|_{L^2(\Omega)}$ (with again $(I_{P,h} p)_{|T} \coloneq I_{P,T} p$ for all $T \in \mathcal{T}_h$).
For the velocity, we additionally consider the errors between the exact solution and the global reconstruction $r_h \underline{u}_h$ such that $(r_h \underline{u}_h)_{|T} \coloneq r_T \underline{u}_T$ for all $T \in \mathcal{T}_h$.
The results collected in Tables~\ref{tab:botti-massa}, \ref{tab:rtn-pk-pk-pk}, and \ref{tab:polygonal} show that these error quantities converge at the expected rates, i.e., $H^1$-like errors on the velocity and $L^2$-like errors on the pressure as $h^{k+1}$, and $L^2$-like errors on the velocity as $h^{k+2}$.
A slight superconvergence of the $L^2$-errors on the pressure (with a rate of $\approx 1.4$ instead of $1$) is observed for the lowest-order case $k = 0$ for all the considered methods.
Numerical results not reported here for the sake of brevity also show that $H^1$-like errors on the pressure converge as $h^k$ (again with a slight superconvergence for $k = 0$).
For the methods of Sections~\ref{sec:applications:botti-massa} and \ref{sec:applications:rtn-pk-pk-pk}, we check pressure-robustness by letting the viscosity $\nu$ vary across 9 orders of magnitude, from $1$ down to $10^{-9}$.
As expected (cf. Remark~\ref{rem:pressure.robustness}), the results in Tables~\ref{tab:botti-massa} and \ref{tab:rtn-pk-pk-pk} show that the errors on the velocity are unaffected from the variations of the viscosity. Indeed, as remarked in \cite{Lederer.Linke.ea:17}, the irrotational and divergence-free part of the momentum forcing term $f$ dominates for $\nu<1$ and $\nu>1$, respectively. 
Notice that, in order to achieve pressure robustness in the vanishing viscosity limit, we needed to increase the degree of exactness of quadrature rules employed for the numerical integration of the forcing term $f$. 
We have also numerically tested the variant of Remark~\ref{rem:pkpo-pk-pk-pk}, which gives results that are essentially identical to those of Table~\ref{tab:rtn-pk-pk-pk}, and are therefore not reported in detail for the sake of conciseness.

\begin{figure}\centering
  \begin{minipage}{0.225\textwidth}\centering
    \includegraphics[height=3.25cm]{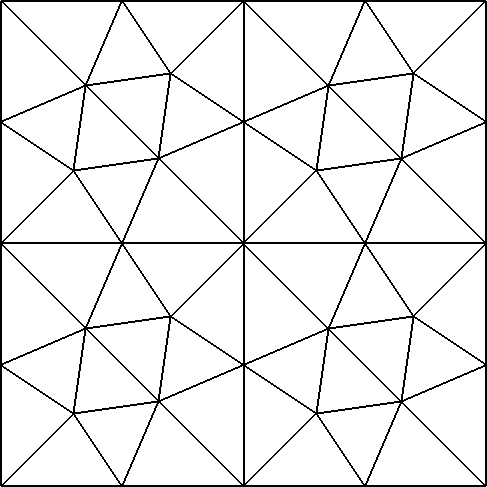}
    \subcaption{``tri''\label{fig:meshes:tri}}
  \end{minipage}
  \begin{minipage}{0.225\textwidth}\centering
    \includegraphics[height=3.25cm]{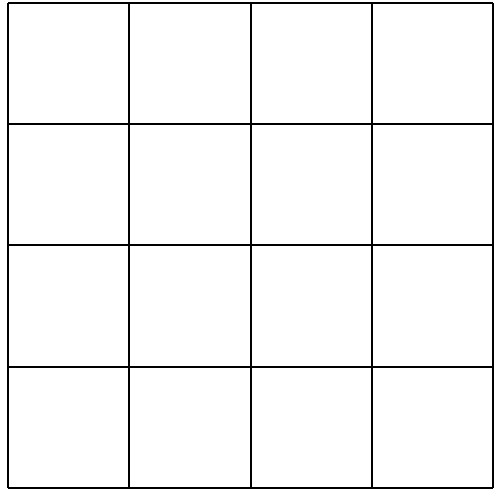}
    \subcaption{``cart''\label{fig:meshes:cart}}
  \end{minipage}    
  \begin{minipage}{0.225\textwidth}\centering
    \includegraphics[height=3.25cm]{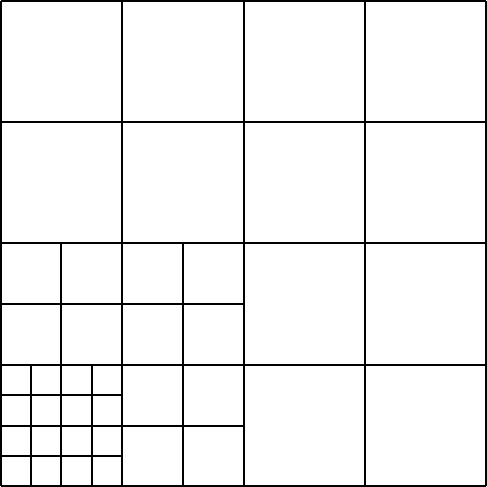}
    \subcaption{``locref''\label{fig:meshes:locref}}
  \end{minipage}        
  \begin{minipage}{0.225\textwidth}\centering
    \includegraphics[height=3.25cm]{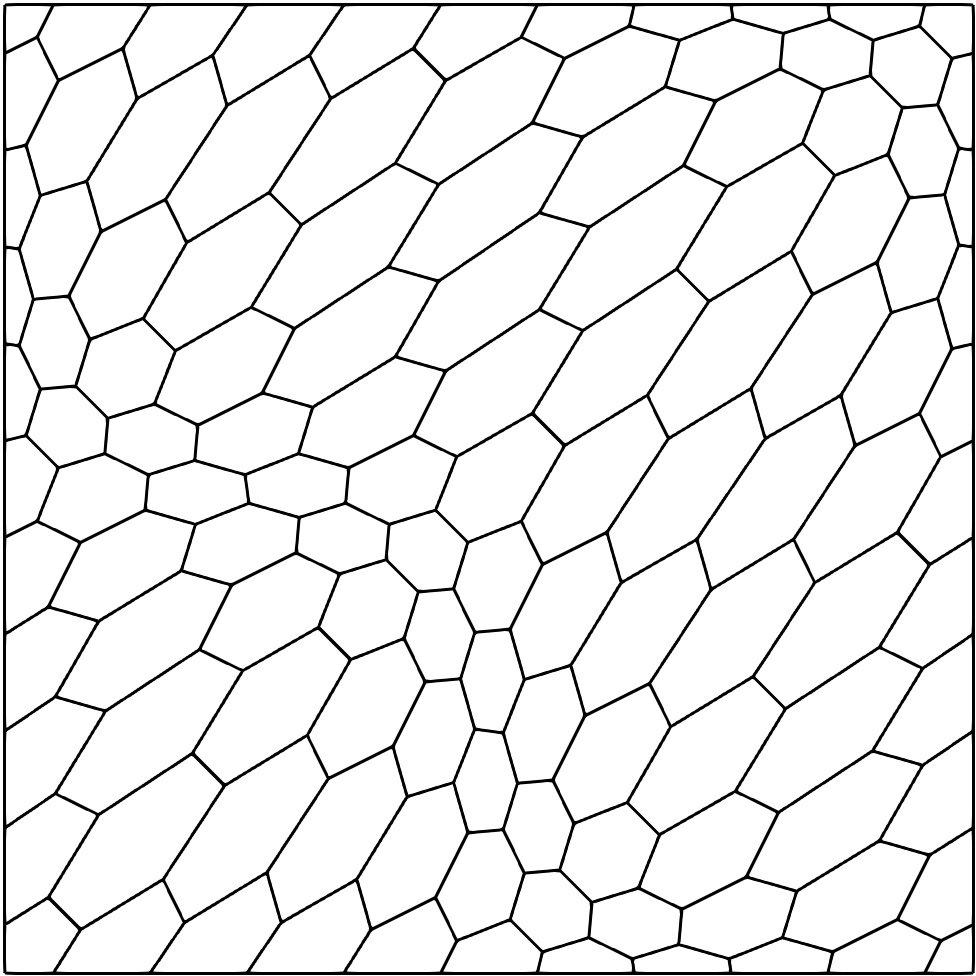}
    \subcaption{``hexa''\label{fig:meshes:hexa}}
  \end{minipage}
  \caption{Mesh families used in the numerical tests of Section~\ref{sec:numerical.examples}.\label{fig:meshes}}
\end{figure}

\begin{table}
  \begin{adjustbox}{width=1\textwidth}
    \begin{tabular}{ccccccccccccc}
  \toprule
  $h$
  & $k$
  & System size
  & $\| \underline{u}_h - \underline{I}_{U,h} u \|_{1,h}$ & OCV
  & $\| \nabla_h (r_h \underline{u}_h - u) \|_{L^2(\Omega)^{d\times d}}$ & OCV
  & $\| u_h - I_{U,h} u \|_{L^2(\Omega)^d}$ & OCV
  & $\| r_h \underline{u}_h - u \|_{L^2(\Omega)^d}$ & OCV
  & $\| p_h - I_{P,h} p \|_{L^2(\Omega)}$ & OCV
  \\
  \midrule
\multicolumn{13}{c}{$ \nu = 1 $} \\ 
\midrule
0.25 & 0 & 393& 3.222194e-02 & --& 2.358650e-02 & --& 3.157105e-03 & --& 1.079341e-03 & --& 8.353185e-03 & --\\ 
0.125 & 0 & 1569& 1.568367e-02 & 1.04& 1.176665e-02 & 1.00& 8.130095e-04 & 1.96& 3.173216e-04 & 1.77& 3.715999e-03 & 1.17\\ 
0.0625 & 0 & 6273& 7.890606e-03 & 0.99& 5.860134e-03 & 1.01& 2.087535e-04 & 1.96& 9.097616e-05 & 1.80& 1.359228e-03 & 1.45\\ 
0.03125 & 0 & 25089& 3.978745e-03 & 0.99& 2.918938e-03 & 1.01& 5.300457e-05 & 1.98& 2.430899e-05 & 1.90& 4.747372e-04 & 1.52\\ 
0.015625 & 0 & 100353& 1.999853e-03 & 0.99& 1.456954e-03 & 1.00& 1.335882e-05 & 1.99& 6.269384e-06 & 1.96& 1.753818e-04 & 1.44\\ 
\midrule
0.25 & 1 & 637& 1.261567e-02 & --& 4.381984e-03 & --& 4.005323e-04 & --& 1.423174e-04 & --& 2.270947e-03 & --\\ 
0.125 & 1 & 2561& 3.012656e-03 & 2.07& 1.239913e-03 & 1.82& 5.094272e-05 & 2.97& 2.011596e-05 & 2.82& 6.069516e-04 & 1.90\\ 
0.0625 & 1 & 10273& 7.423158e-04 & 2.02& 3.140934e-04 & 1.98& 6.487856e-06 & 2.97& 2.558775e-06 & 2.97& 1.515679e-04 & 2.00\\ 
0.03125 & 1 & 41153& 1.846133e-04 & 2.01& 7.882465e-05 & 1.99& 8.261358e-07 & 2.97& 3.225663e-07 & 2.99& 3.765641e-05 & 2.01\\ 
0.015625 & 1 & 164737& 4.606087e-05 & 2.00& 1.974435e-05 & 2.00& 1.046304e-07 & 2.98& 4.050395e-08 & 2.99& 9.359506e-06 & 2.01\\ 
\midrule
0.25 & 2 & 881& 4.901561e-03 & --& 6.467589e-04 & --& 7.339555e-05 & --& 1.163225e-05 & --& 3.628231e-04 & --\\ 
0.125 & 2 & 3553& 5.784949e-04 & 3.08& 8.006862e-05 & 3.01& 4.193709e-06 & 4.13& 7.971323e-07 & 3.87& 4.640113e-05 & 2.97\\ 
0.0625 & 2 & 14273& 6.973436e-05 & 3.05& 9.896615e-06 & 3.02& 2.512826e-07 & 4.06& 5.200708e-08 & 3.94& 5.795639e-06 & 3.00\\ 
0.03125 & 2 & 57217& 8.605670e-06 & 3.02& 1.246030e-06 & 2.99& 1.554178e-08 & 4.02& 3.388334e-09 & 3.94& 7.286071e-07 & 2.99\\ 
0.015625 & 2 & 229121& 1.070946e-06 & 3.01& 1.570128e-07 & 2.99& 9.699535e-10 & 4.00& 2.175903e-10 & 3.96& 9.142654e-08 & 2.99\\ 
\midrule
\multicolumn{13}{c}{$ \nu = 0.001 $} \\ 
\midrule
0.25 & 0 & 393& 3.222194e-02 & --& 2.358650e-02 & --& 3.157105e-03 & --& 1.079341e-03 & --& 8.353185e-06 & --\\ 
0.125 & 0 & 1569& 1.568367e-02 & 1.04& 1.176665e-02 & 1.00& 8.130095e-04 & 1.96& 3.173216e-04 & 1.77& 3.715999e-06 & 1.17\\ 
0.0625 & 0 & 6273& 7.890606e-03 & 0.99& 5.860134e-03 & 1.01& 2.087535e-04 & 1.96& 9.097616e-05 & 1.80& 1.359228e-06 & 1.45\\ 
0.03125 & 0 & 25089& 3.978745e-03 & 0.99& 2.918938e-03 & 1.01& 5.300457e-05 & 1.98& 2.430899e-05 & 1.90& 4.747372e-07 & 1.52\\ 
0.015625 & 0 & 100353& 1.999853e-03 & 0.99& 1.456954e-03 & 1.00& 1.335882e-05 & 1.99& 6.269384e-06 & 1.96& 1.753818e-07 & 1.44\\ 
\midrule
0.25 & 1 & 637& 1.261567e-02 & --& 4.381984e-03 & --& 4.005323e-04 & --& 1.423174e-04 & --& 2.270947e-06 & --\\ 
0.125 & 1 & 2561& 3.012656e-03 & 2.07& 1.239913e-03 & 1.82& 5.094272e-05 & 2.97& 2.011596e-05 & 2.82& 6.069516e-07 & 1.90\\ 
0.0625 & 1 & 10273& 7.423158e-04 & 2.02& 3.140934e-04 & 1.98& 6.487856e-06 & 2.97& 2.558775e-06 & 2.97& 1.515679e-07 & 2.00\\ 
0.03125 & 1 & 41153& 1.846133e-04 & 2.01& 7.882465e-05 & 1.99& 8.261358e-07 & 2.97& 3.225663e-07 & 2.99& 3.765641e-08 & 2.01\\ 
0.015625 & 1 & 164737& 4.606087e-05 & 2.00& 1.974435e-05 & 2.00& 1.046304e-07 & 2.98& 4.050395e-08 & 2.99& 9.359506e-09 & 2.01\\ 
\midrule
0.25 & 2 & 881& 4.901561e-03 & --& 6.467589e-04 & --& 7.339555e-05 & --& 1.163225e-05 & --& 3.628231e-07 & --\\ 
0.125 & 2 & 3553& 5.784949e-04 & 3.08& 8.006862e-05 & 3.01& 4.193709e-06 & 4.13& 7.971323e-07 & 3.87& 4.640113e-08 & 2.97\\ 
0.0625 & 2 & 14273& 6.973436e-05 & 3.05& 9.896615e-06 & 3.02& 2.512826e-07 & 4.06& 5.200708e-08 & 3.94& 5.795639e-09 & 3.00\\ 
0.03125 & 2 & 57217& 8.605670e-06 & 3.02& 1.246030e-06 & 2.99& 1.554178e-08 & 4.02& 3.388334e-09 & 3.94& 7.286071e-10 & 2.99\\ 
0.015625 & 2 & 229121& 1.070946e-06 & 3.01& 1.570128e-07 & 2.99& 9.699535e-10 & 4.00& 2.175903e-10 & 3.96& 9.142654e-11 & 2.99\\ 
\midrule
\multicolumn{13}{c}{$ \nu = 1e-06 $} \\ 
\midrule
0.25 & 0 & 393& 3.222194e-02 & --& 2.358650e-02 & --& 3.157105e-03 & --& 1.079341e-03 & --& 8.353185e-09 & --\\ 
0.125 & 0 & 1569& 1.568367e-02 & 1.04& 1.176665e-02 & 1.00& 8.130095e-04 & 1.96& 3.173216e-04 & 1.77& 3.715999e-09 & 1.17\\ 
0.0625 & 0 & 6273& 7.890606e-03 & 0.99& 5.860134e-03 & 1.01& 2.087535e-04 & 1.96& 9.097616e-05 & 1.80& 1.359228e-09 & 1.45\\ 
0.03125 & 0 & 25089& 3.978745e-03 & 0.99& 2.918938e-03 & 1.01& 5.300457e-05 & 1.98& 2.430899e-05 & 1.90& 4.747372e-10 & 1.52\\ 
0.015625 & 0 & 100353& 1.999853e-03 & 0.99& 1.456954e-03 & 1.00& 1.335882e-05 & 1.99& 6.269385e-06 & 1.96& 1.753818e-10 & 1.44\\ 
\midrule
0.25 & 1 & 637& 1.261567e-02 & --& 4.381984e-03 & --& 4.005323e-04 & --& 1.423174e-04 & --& 2.270947e-09 & --\\ 
0.125 & 1 & 2561& 3.012656e-03 & 2.07& 1.239913e-03 & 1.82& 5.094272e-05 & 2.97& 2.011596e-05 & 2.82& 6.069516e-10 & 1.90\\ 
0.0625 & 1 & 10273& 7.423158e-04 & 2.02& 3.140934e-04 & 1.98& 6.487856e-06 & 2.97& 2.558775e-06 & 2.97& 1.515679e-10 & 2.00\\ 
0.03125 & 1 & 41153& 1.846133e-04 & 2.01& 7.882465e-05 & 1.99& 8.261358e-07 & 2.97& 3.225663e-07 & 2.99& 3.765642e-11 & 2.01\\ 
0.015625 & 1 & 164737& 4.606087e-05 & 2.00& 1.974435e-05 & 2.00& 1.046304e-07 & 2.98& 4.050395e-08 & 2.99& 9.359491e-12 & 2.01\\ 
\midrule
0.25 & 2 & 881& 4.901561e-03 & --& 6.467589e-04 & --& 7.339556e-05 & --& 1.163225e-05 & --& 3.628230e-10 & --\\ 
0.125 & 2 & 3553& 5.784949e-04 & 3.08& 8.006862e-05 & 3.01& 4.193709e-06 & 4.13& 7.971323e-07 & 3.87& 4.640111e-11 & 2.97\\ 
0.0625 & 2 & 14273& 6.973438e-05 & 3.05& 9.896615e-06 & 3.02& 2.512827e-07 & 4.06& 5.200708e-08 & 3.94& 5.795651e-12 & 3.00\\ 
0.03125 & 2 & 57217& 8.605698e-06 & 3.02& 1.246029e-06 & 2.99& 1.554182e-08 & 4.02& 3.388331e-09 & 3.94& 7.285984e-13 & 2.99\\ 
0.015625 & 2 & 229121& 1.072212e-06 & 3.00& 1.570176e-07 & 2.99& 9.707280e-10 & 4.00& 2.176095e-10 & 3.96& 9.160016e-14 & 2.99\\ 
\midrule
\multicolumn{13}{c}{$ \nu = 1e-09 $} \\ 
\midrule
0.25 & 0 & 393& 3.222197e-02 & --& 2.358650e-02 & --& 3.157101e-03 & --& 1.079343e-03 & --& 8.353126e-12 & --\\ 
0.125 & 0 & 1569& 1.568351e-02 & 1.04& 1.176655e-02 & 1.00& 8.128663e-04 & 1.96& 3.172635e-04 & 1.77& 3.716196e-12 & 1.17\\ 
0.0625 & 0 & 6273& 7.890564e-03 & 0.99& 5.860114e-03 & 1.01& 2.086862e-04 & 1.96& 9.094945e-05 & 1.80& 1.359335e-12 & 1.45\\ 
0.03125 & 0 & 25089& 3.978735e-03 & 0.99& 2.918933e-03 & 1.01& 5.297332e-05 & 1.98& 2.429646e-05 & 1.90& 4.747921e-13 & 1.52\\ 
0.015625 & 0 & 100353& 1.999850e-03 & 0.99& 1.456953e-03 & 1.00& 1.334166e-05 & 1.99& 6.262585e-06 & 1.96& 1.754138e-13 & 1.44\\ 
\midrule
0.25 & 1 & 637& 1.261571e-02 & --& 4.381978e-03 & --& 4.005307e-04 & --& 1.423153e-04 & --& 2.270990e-12 & --\\ 
0.125 & 1 & 2561& 3.012621e-03 & 2.07& 1.239912e-03 & 1.82& 5.093999e-05 & 2.98& 2.011354e-05 & 2.82& 6.069569e-13 & 1.90\\ 
0.0625 & 1 & 10273& 7.422844e-04 & 2.02& 3.140926e-04 & 1.98& 6.486999e-06 & 2.97& 2.558201e-06 & 2.97& 1.515955e-13 & 2.00\\ 
0.03125 & 1 & 41153& 1.849331e-04 & 2.00& 7.882788e-05 & 1.99& 8.264979e-07 & 2.97& 3.225103e-07 & 2.99& 3.774118e-14 & 2.01\\ 
0.015625 & 1 & 164737& 4.819223e-05 & 1.94& 1.975600e-05 & 2.00& 1.062163e-07 & 2.96& 4.051995e-08 & 2.99& 9.917517e-15 & 1.93\\ 
\midrule
0.25 & 2 & 881& 4.901606e-03 & --& 6.467540e-04 & --& 7.339583e-05 & --& 1.163162e-05 & --& 3.628306e-13 & --\\ 
0.125 & 2 & 3553& 5.784848e-04 & 3.08& 8.006882e-05 & 3.01& 4.193606e-06 & 4.13& 7.971120e-07 & 3.87& 4.645136e-14 & 2.97\\ 
0.0625 & 2 & 14273& 7.017755e-05 & 3.04& 9.899016e-06 & 3.02& 2.526041e-07 & 4.05& 5.206829e-08 & 3.94& 6.384613e-15 & 2.86\\ 
0.03125 & 2 & 57217& 1.670814e-05 & 2.07& 1.294704e-06 & 2.93& 2.775679e-08 & 3.19& 3.717234e-09 & 3.81& 3.168766e-15 & 1.01\\ 
0.015625 & 2 & 229121& 5.230319e-05 & -1.65& 1.214343e-06 & 0.09& 3.896141e-08 & -0.49& 1.919767e-09 & 0.95& 5.651455e-15 & -0.83\\ 
\bottomrule
\end{tabular}
  \end{adjustbox}
  \caption{Numerical results for the Botti--Massa method of Section~\ref{sec:applications:botti-massa}.\label{tab:botti-massa}}
\end{table}

\begin{table}
  \begin{adjustbox}{width=1\textwidth}
    \begin{tabular}{ccccccccccccc}
  \toprule
  $h$
  & $k$
  & System size
  & $\| \underline{u}_h - \underline{I}_{U,h} u \|_{1,h}$ & OCV
  & $\| \nabla_h (r_h \underline{u}_h - u) \|_{L^2(\Omega)^{d\times d}}$ & OCV
  & $\| u_h - I_{U,h} u \|_{L^2(\Omega)^d}$ & OCV
  & $\| r_h \underline{u}_h - u \|_{L^2(\Omega)^d}$ & OCV
  & $\| p_h - I_{P,h} p \|_{L^2(\Omega)}$ & OCV
  \\
  \midrule
\multicolumn{13}{c}{$ \nu = 1 $} \\ 
\midrule
0.25 & 0 & 301& 5.091366e-02 & --& 3.763716e-02 & --& 5.452321e-03 & --& 2.192226e-03 & --& 1.057235e-02 & --\\ 
0.125 & 0 & 1217& 2.664248e-02 & 0.93& 1.953002e-02 & 0.95& 1.438660e-03 & 1.92& 6.349527e-04 & 1.79& 5.209898e-03 & 1.02\\ 
0.0625 & 0 & 4897& 1.371812e-02 & 0.96& 9.874165e-03 & 0.98& 3.749484e-04 & 1.94& 1.811691e-04 & 1.81& 1.987649e-03 & 1.39\\ 
0.03125 & 0 & 19649& 6.961800e-03 & 0.98& 4.949384e-03 & 1.00& 9.571999e-05 & 1.97& 4.831682e-05 & 1.91& 6.963471e-04 & 1.51\\ 
0.015625 & 0 & 78721& 3.504377e-03 & 0.99& 2.476364e-03 & 1.00& 2.416339e-05 & 1.99& 1.243355e-05 & 1.96& 2.562386e-04 & 1.44\\ 
\midrule
0.25 & 1 & 545& 1.260873e-02 & --& 5.067568e-03 & --& 4.695794e-04 & --& 1.661224e-04 & --& 2.574142e-03 & --\\ 
0.125 & 1 & 2209& 3.495394e-03 & 1.85& 1.437499e-03 & 1.82& 6.391980e-05 & 2.88& 2.394042e-05 & 2.79& 7.023496e-04 & 1.87\\ 
0.0625 & 1 & 8897& 9.150857e-04 & 1.93& 3.657393e-04 & 1.97& 8.347818e-06 & 2.94& 3.061214e-06 & 2.97& 1.747567e-04 & 2.01\\ 
0.03125 & 1 & 35713& 2.353234e-04 & 1.96& 9.211973e-05 & 1.99& 1.076894e-06 & 2.95& 3.867520e-07 & 2.98& 4.309645e-05 & 2.02\\ 
0.015625 & 1 & 143105& 5.983437e-05 & 1.98& 2.312937e-05 & 1.99& 1.373785e-07 & 2.97& 4.860589e-08 & 2.99& 1.063320e-05 & 2.02\\ 
\midrule
0.25 & 2 & 789& 1.989339e-03 & --& 7.215444e-04 & --& 4.092264e-05 & --& 1.437024e-05 & --& 5.024945e-04 & --\\ 
0.125 & 2 & 3201& 1.714420e-04 & 3.54& 9.153136e-05 & 2.98& 1.755080e-06 & 4.54& 1.012735e-06 & 3.83& 6.278277e-05 & 3.00\\ 
0.0625 & 2 & 12897& 1.816462e-05 & 3.24& 1.149733e-05 & 2.99& 9.536951e-08 & 4.20& 6.728081e-08 & 3.91& 7.930655e-06 & 2.98\\ 
0.03125 & 2 & 51777& 2.172452e-06 & 3.06& 1.459545e-06 & 2.98& 5.795572e-09 & 4.04& 4.417584e-09 & 3.93& 1.011191e-06 & 2.97\\ 
0.015625 & 2 & 207489& 2.715882e-07 & 3.00& 1.846122e-07 & 2.98& 3.642384e-10 & 3.99& 2.845126e-10 & 3.96& 1.280167e-07 & 2.98\\ 
\midrule
\multicolumn{13}{c}{$ \nu = 0.001 $} \\ 
\midrule
0.25 & 0 & 301& 5.091366e-02 & --& 3.763716e-02 & --& 5.452321e-03 & --& 2.192226e-03 & --& 1.057235e-05 & --\\ 
0.125 & 0 & 1217& 2.664248e-02 & 0.93& 1.953002e-02 & 0.95& 1.438660e-03 & 1.92& 6.349527e-04 & 1.79& 5.209898e-06 & 1.02\\ 
0.0625 & 0 & 4897& 1.371812e-02 & 0.96& 9.874165e-03 & 0.98& 3.749484e-04 & 1.94& 1.811691e-04 & 1.81& 1.987649e-06 & 1.39\\ 
0.03125 & 0 & 19649& 6.961800e-03 & 0.98& 4.949384e-03 & 1.00& 9.571999e-05 & 1.97& 4.831682e-05 & 1.91& 6.963471e-07 & 1.51\\ 
0.015625 & 0 & 78721& 3.504377e-03 & 0.99& 2.476364e-03 & 1.00& 2.416339e-05 & 1.99& 1.243355e-05 & 1.96& 2.562386e-07 & 1.44\\ 
\midrule
0.25 & 1 & 545& 1.260873e-02 & --& 5.067568e-03 & --& 4.695794e-04 & --& 1.661224e-04 & --& 2.574142e-06 & --\\ 
0.125 & 1 & 2209& 3.495394e-03 & 1.85& 1.437499e-03 & 1.82& 6.391980e-05 & 2.88& 2.394042e-05 & 2.79& 7.023496e-07 & 1.87\\ 
0.0625 & 1 & 8897& 9.150857e-04 & 1.93& 3.657393e-04 & 1.97& 8.347818e-06 & 2.94& 3.061214e-06 & 2.97& 1.747567e-07 & 2.01\\ 
0.03125 & 1 & 35713& 2.353234e-04 & 1.96& 9.211973e-05 & 1.99& 1.076894e-06 & 2.95& 3.867520e-07 & 2.98& 4.309645e-08 & 2.02\\ 
0.015625 & 1 & 143105& 5.983437e-05 & 1.98& 2.312937e-05 & 1.99& 1.373785e-07 & 2.97& 4.860589e-08 & 2.99& 1.063320e-08 & 2.02\\ 
\midrule
0.25 & 2 & 789& 1.989339e-03 & --& 7.215444e-04 & --& 4.092264e-05 & --& 1.437024e-05 & --& 5.024945e-07 & --\\ 
0.125 & 2 & 3201& 1.714420e-04 & 3.54& 9.153136e-05 & 2.98& 1.755080e-06 & 4.54& 1.012735e-06 & 3.83& 6.278277e-08 & 3.00\\ 
0.0625 & 2 & 12897& 1.816462e-05 & 3.24& 1.149733e-05 & 2.99& 9.536951e-08 & 4.20& 6.728081e-08 & 3.91& 7.930655e-09 & 2.98\\ 
0.03125 & 2 & 51777& 2.172452e-06 & 3.06& 1.459545e-06 & 2.98& 5.795572e-09 & 4.04& 4.417584e-09 & 3.93& 1.011191e-09 & 2.97\\ 
0.015625 & 2 & 207489& 2.715882e-07 & 3.00& 1.846122e-07 & 2.98& 3.642384e-10 & 3.99& 2.845126e-10 & 3.96& 1.280167e-10 & 2.98\\ 
\midrule
\multicolumn{13}{c}{$ \nu = 1e-06 $} \\ 
\midrule
0.25 & 0 & 301& 5.091366e-02 & --& 3.763716e-02 & --& 5.452321e-03 & --& 2.192226e-03 & --& 1.057235e-08 & --\\ 
0.125 & 0 & 1217& 2.664248e-02 & 0.93& 1.953002e-02 & 0.95& 1.438660e-03 & 1.92& 6.349527e-04 & 1.79& 5.209898e-09 & 1.02\\ 
0.0625 & 0 & 4897& 1.371812e-02 & 0.96& 9.874165e-03 & 0.98& 3.749484e-04 & 1.94& 1.811691e-04 & 1.81& 1.987649e-09 & 1.39\\ 
0.03125 & 0 & 19649& 6.961800e-03 & 0.98& 4.949384e-03 & 1.00& 9.571999e-05 & 1.97& 4.831682e-05 & 1.91& 6.963471e-10 & 1.51\\ 
0.015625 & 0 & 78721& 3.504377e-03 & 0.99& 2.476364e-03 & 1.00& 2.416339e-05 & 1.99& 1.243355e-05 & 1.96& 2.562386e-10 & 1.44\\ 
\midrule
0.25 & 1 & 545& 1.260873e-02 & --& 5.067568e-03 & --& 4.695794e-04 & --& 1.661224e-04 & --& 2.574142e-09 & --\\ 
0.125 & 1 & 2209& 3.495394e-03 & 1.85& 1.437499e-03 & 1.82& 6.391980e-05 & 2.88& 2.394042e-05 & 2.79& 7.023496e-10 & 1.87\\ 
0.0625 & 1 & 8897& 9.150857e-04 & 1.93& 3.657393e-04 & 1.97& 8.347818e-06 & 2.94& 3.061214e-06 & 2.97& 1.747568e-10 & 2.01\\ 
0.03125 & 1 & 35713& 2.353234e-04 & 1.96& 9.211973e-05 & 1.99& 1.076894e-06 & 2.95& 3.867520e-07 & 2.98& 4.309646e-11 & 2.02\\ 
0.015625 & 1 & 143105& 5.983437e-05 & 1.98& 2.312937e-05 & 1.99& 1.373785e-07 & 2.97& 4.860590e-08 & 2.99& 1.063319e-11 & 2.02\\ 
\midrule
0.25 & 2 & 789& 1.989339e-03 & --& 7.215444e-04 & --& 4.092264e-05 & --& 1.437023e-05 & --& 5.024945e-10 & --\\ 
0.125 & 2 & 3201& 1.714420e-04 & 3.54& 9.153136e-05 & 2.98& 1.755080e-06 & 4.54& 1.012735e-06 & 3.83& 6.278276e-11 & 3.00\\ 
0.0625 & 2 & 12897& 1.816462e-05 & 3.24& 1.149733e-05 & 2.99& 9.536949e-08 & 4.20& 6.728081e-08 & 3.91& 7.930667e-12 & 2.98\\ 
0.03125 & 2 & 51777& 2.172451e-06 & 3.06& 1.459544e-06 & 2.98& 5.795567e-09 & 4.04& 4.417579e-09 & 3.93& 1.011182e-12 & 2.97\\ 
0.015625 & 2 & 207489& 2.717105e-07 & 3.00& 1.846151e-07 & 2.98& 3.643416e-10 & 3.99& 2.845273e-10 & 3.96& 1.281173e-13 & 2.98\\ 
\midrule
\multicolumn{13}{c}{$ \nu = 1e-09 $} \\ 
\midrule
0.25 & 0 & 301& 5.091363e-02 & --& 3.763719e-02 & --& 5.452310e-03 & --& 2.192225e-03 & --& 1.057255e-11 & --\\ 
0.125 & 0 & 1217& 2.664234e-02 & 0.93& 1.952987e-02 & 0.95& 1.438559e-03 & 1.92& 6.348927e-04 & 1.79& 5.210224e-12 & 1.02\\ 
0.0625 & 0 & 4897& 1.371811e-02 & 0.96& 9.874163e-03 & 0.98& 3.749366e-04 & 1.94& 1.811630e-04 & 1.81& 1.987694e-12 & 1.39\\ 
0.03125 & 0 & 19649& 6.961792e-03 & 0.98& 4.949376e-03 & 1.00& 9.569553e-05 & 1.97& 4.830410e-05 & 1.91& 6.963895e-13 & 1.51\\ 
0.015625 & 0 & 78721& 3.504375e-03 & 0.99& 2.476363e-03 & 1.00& 2.415342e-05 & 1.99& 1.242844e-05 & 1.96& 2.562572e-13 & 1.44\\ 
\midrule
0.25 & 1 & 545& 1.260861e-02 & --& 5.067566e-03 & --& 4.695673e-04 & --& 1.661151e-04 & --& 2.574172e-12 & --\\ 
0.125 & 1 & 2209& 3.495056e-03 & 1.85& 1.437494e-03 & 1.82& 6.390885e-05 & 2.88& 2.393561e-05 & 2.79& 7.023477e-13 & 1.87\\ 
0.0625 & 1 & 8897& 9.149967e-04 & 1.93& 3.657387e-04 & 1.97& 8.346386e-06 & 2.94& 3.060646e-06 & 2.97& 1.747808e-13 & 2.01\\ 
0.03125 & 1 & 35713& 2.353216e-04 & 1.96& 9.212023e-05 & 1.99& 1.076851e-06 & 2.95& 3.867169e-07 & 2.98& 4.316424e-14 & 2.02\\ 
0.015625 & 1 & 143105& 5.989604e-05 & 1.97& 2.313613e-05 & 1.99& 1.374942e-07 & 2.97& 4.863242e-08 & 2.99& 1.107310e-14 & 1.96\\ 
\midrule
0.25 & 2 & 789& 1.989319e-03 & --& 7.215433e-04 & --& 4.092234e-05 & --& 1.437001e-05 & --& 5.025068e-13 & --\\ 
0.125 & 2 & 3201& 1.714564e-04 & 3.54& 9.153210e-05 & 2.98& 1.755206e-06 & 4.54& 1.012769e-06 & 3.83& 6.282424e-14 & 3.00\\ 
0.0625 & 2 & 12897& 1.826850e-05 & 3.23& 1.149849e-05 & 2.99& 9.575826e-08 & 4.20& 6.730105e-08 & 3.91& 8.367764e-15 & 2.91\\ 
0.03125 & 2 & 51777& 3.348559e-06 & 2.45& 1.494559e-06 & 2.94& 8.286685e-09 & 3.53& 4.796850e-09 & 3.81& 3.193360e-15 & 1.39\\ 
0.015625 & 2 & 207489& 8.084852e-06 & -1.27& 9.689466e-07 & 0.63& 8.450114e-09 & -0.03& 2.358977e-09 & 1.02& 5.357670e-15 & -0.75\\ 
\bottomrule
\end{tabular}
  \end{adjustbox}
  \caption{Numerical results for the new method with Raviart--Thomas--N\'ed\'elec velocities at elements of Section~\ref{sec:applications:rtn-pk-pk-pk}.\label{tab:rtn-pk-pk-pk}}
\end{table}


\begin{table}
  \begin{adjustbox}{width=1\textwidth}
    \begin{tabular}{ccccccccccccc}
  \toprule
  $h$
  & $k$
  & System size
  & $\| \underline{u}_h - \underline{I}_{U,h} u \|_{1,h}$ & OCV
  & $\| \nabla_h (r_h \underline{u}_h - u) \|_{L^2(\Omega)^{d\times d}}$ & OCV
  & $\| u_h - I_{U,h} u \|_{L^2(\Omega)^d}$ & OCV
  & $\| r_h \underline{u}_h - u \|_{L^2(\Omega)^d}$ & OCV
  & $\| p_h - I_{P,h} p \|_{L^2(\Omega)}$ & OCV
  \\
  \midrule
\multicolumn{13}{c}{``tri'' mesh family} \\ 
\midrule
0.25 & 0 & 301& 2.871509e-01 & --& 1.388025e-01 & --& 2.593695e-02 & --& 1.600589e-02 & --& 1.181657e-01 & --\\ 
0.125 & 0 & 1217& 1.777113e-01 & 0.69& 8.032856e-02 & 0.79& 8.742641e-03 & 1.57& 6.711338e-03 & 1.25& 7.136302e-02 & 0.73\\ 
0.0625 & 0 & 4897& 1.056728e-01 & 0.75& 4.066457e-02 & 0.98& 2.754628e-03 & 1.67& 2.345353e-03 & 1.52& 3.546204e-02 & 1.01\\ 
0.03125 & 0 & 19649& 5.919433e-02 & 0.84& 1.909566e-02 & 1.09& 7.961755e-04 & 1.79& 7.104760e-04 & 1.72& 1.499777e-02 & 1.24\\ 
0.015625 & 0 & 78721& 3.155741e-02 & 0.91& 8.815891e-03 & 1.12& 2.154894e-04 & 1.89& 1.964997e-04 & 1.85& 5.693412e-03 & 1.40\\ 
\midrule
0.25 & 1 & 545& 1.094570e-01 & --& 1.442178e-02 & --& 3.566780e-03 & --& 5.455960e-04 & --& 9.971271e-03 & --\\ 
0.125 & 1 & 2209& 2.864697e-02 & 1.93& 4.029509e-03 & 1.84& 4.708069e-04 & 2.92& 7.396471e-05 & 2.88& 2.444760e-03 & 2.03\\ 
0.0625 & 1 & 8897& 7.259557e-03 & 1.98& 1.048116e-03 & 1.94& 5.981108e-05 & 2.98& 9.481457e-06 & 2.96& 5.908448e-04 & 2.05\\ 
0.03125 & 1 & 35713& 1.824671e-03 & 1.99& 2.671313e-04 & 1.97& 7.525256e-06 & 2.99& 1.202422e-06 & 2.98& 1.452048e-04 & 2.02\\ 
0.015625 & 1 & 143105& 4.572920e-04 & 2.00& 6.745664e-05 & 1.99& 9.435208e-07 & 3.00& 1.516119e-07 & 2.99& 3.603847e-05 & 2.01\\ 
\midrule
0.25 & 2 & 789& 2.229905e-02 & --& 1.419633e-03 & --& 4.267495e-04 & --& 3.575930e-05 & --& 1.104445e-03 & --\\ 
0.125 & 2 & 3201& 2.913007e-03 & 2.94& 2.057320e-04 & 2.79& 2.782726e-05 & 3.94& 2.784274e-06 & 3.68& 1.603266e-04 & 2.78\\ 
0.0625 & 2 & 12897& 3.697032e-04 & 2.98& 2.803240e-05 & 2.88& 1.766127e-06 & 3.98& 1.983784e-07 & 3.81& 2.182177e-05 & 2.88\\ 
0.03125 & 2 & 51777& 4.652907e-05 & 2.99& 3.679046e-06 & 2.93& 1.112016e-07 & 3.99& 1.332765e-08 & 3.90& 2.855088e-06 & 2.93\\ 
0.015625 & 2 & 207489& 5.835491e-06 & 3.00& 4.719275e-07 & 2.96& 6.976227e-09 & 3.99& 8.650121e-10 & 3.95& 3.653625e-07 & 2.97\\ 
\midrule
\multicolumn{13}{c}{``cart'' mesh family} \\ 
\midrule
0.141421 & 0 & 681& 7.431549e-02 & --& 5.434746e-02 & --& 8.757928e-03 & --& 6.594299e-03 & --& 6.380900e-02 & --\\ 
0.0707107 & 0 & 2761& 3.107123e-02 & 1.26& 2.721935e-02 & 1.00& 2.655586e-03 & 1.72& 2.306187e-03 & 1.52& 3.273492e-02 & 0.96\\ 
0.0353553 & 0 & 11121& 1.241095e-02 & 1.32& 1.190420e-02 & 1.19& 7.555894e-04 & 1.81& 7.039030e-04 & 1.71& 1.413110e-02 & 1.21\\ 
0.0176777 & 0 & 44641& 4.881669e-03 & 1.35& 4.782175e-03 & 1.32& 2.034841e-04 & 1.89& 1.960254e-04 & 1.84& 5.418242e-03 & 1.38\\ 
\midrule
0.141421 & 1 & 1261& 2.231652e-02 & --& 2.071433e-03 & --& 5.267229e-04 & --& 3.493709e-05 & --& 7.189385e-04 & --\\ 
0.0707107 & 1 & 5121& 5.649050e-03 & 1.98& 4.834399e-04 & 2.10& 6.658733e-05 & 2.98& 3.655879e-06 & 3.26& 9.922670e-05 & 2.86\\ 
0.0353553 & 1 & 20641& 1.418751e-03 & 1.99& 1.153138e-04 & 2.07& 8.350360e-06 & 3.00& 3.970154e-07 & 3.20& 1.301183e-05 & 2.93\\ 
0.0176777 & 1 & 82881& 3.553536e-04 & 2.00& 2.813487e-05 & 2.04& 1.044885e-06 & 3.00& 4.589223e-08 & 3.11& 1.674172e-06 & 2.96\\ 
\midrule
0.141421 & 2 & 1841& 2.306263e-03 & --& 2.086652e-04 & --& 2.840043e-05 & --& 2.186706e-06 & --& 1.069296e-04 & --\\ 
0.0707107 & 2 & 7481& 2.916838e-04 & 2.98& 2.723758e-05 & 2.94& 1.793080e-06 & 3.99& 1.453602e-07 & 3.91& 1.532567e-05 & 2.80\\ 
0.0353553 & 2 & 30161& 3.672161e-05 & 2.99& 3.453053e-06 & 2.98& 1.127570e-07 & 3.99& 9.301390e-09 & 3.97& 2.018631e-06 & 2.92\\ 
0.0176777 & 2 & 121121& 4.610043e-06 & 2.99& 4.335188e-07 & 2.99& 7.073963e-09 & 3.99& 5.863440e-10 & 3.99& 2.574813e-07 & 2.97\\ 
\midrule
\multicolumn{13}{c}{``locref'' mesh family} \\ 
\midrule
0.176777 & 0 & 1121& 9.662929e-02 & --& 6.562072e-02 & --& 1.260023e-02 & --& 8.866492e-03 & --& 7.575972e-02 & --\\ 
0.0883883 & 0 & 4481& 4.099609e-02 & 1.24& 3.436386e-02 & 0.93& 3.906079e-03 & 1.69& 3.286328e-03 & 1.43& 4.145393e-02 & 0.87\\ 
0.0441942 & 0 & 17921& 1.652976e-02 & 1.31& 1.560495e-02 & 1.14& 1.135878e-03 & 1.78& 1.043778e-03 & 1.65& 1.880963e-02 & 1.14\\ 
0.0220971 & 0 & 71681& 6.483533e-03 & 1.35& 6.378802e-03 & 1.29& 3.108131e-04 & 1.87& 2.979813e-04 & 1.81& 7.448419e-03 & 1.34\\ 
\midrule
0.176777 & 1 & 2081& 3.419560e-02 & --& 2.922194e-03 & --& 1.009461e-03 & --& 6.113120e-05 & --& 1.309813e-03 & --\\ 
0.0883883 & 1 & 8321& 8.705068e-03 & 1.97& 6.775391e-04 & 2.11& 1.285386e-04 & 2.97& 6.530443e-06 & 3.23& 1.873541e-04 & 2.81\\ 
0.0441942 & 1 & 33281& 2.190150e-03 & 1.99& 1.594929e-04 & 2.09& 1.615289e-05 & 2.99& 6.990477e-07 & 3.22& 2.489209e-05 & 2.91\\ 
0.0220971 & 1 & 133121& 5.489024e-04 & 2.00& 3.859099e-05 & 2.05& 2.022588e-06 & 3.00& 7.936535e-08 & 3.14& 3.227847e-06 & 2.95\\ 
\midrule
0.176777 & 2 & 3041& 4.385363e-03 & --& 3.551913e-04 & --& 6.724034e-05 & --& 4.645085e-06 & --& 1.904914e-04 & --\\ 
0.0883883 & 2 & 12161& 5.582283e-04 & 2.97& 4.762003e-05 & 2.90& 4.278387e-06 & 3.97& 3.194758e-07 & 3.86& 2.847426e-05 & 2.74\\ 
0.0441942 & 2 & 48641& 7.038299e-05 & 2.99& 6.108618e-06 & 2.96& 2.695583e-07 & 3.99& 2.082572e-08 & 3.94& 3.835956e-06 & 2.89\\ 
0.0220971 & 2 & 194561& 8.841014e-06 & 2.99& 7.707985e-07 & 2.99& 1.692408e-08 & 3.99& 1.324258e-09 & 3.98& 4.942612e-07 & 2.96\\ 
\midrule
\multicolumn{13}{c}{``hexa'' mesh family} \\ 
\midrule
0.241412 & 0 & 1162& 5.770683e-02 & --& 6.001934e-02 & --& 8.695028e-03 & --& 7.650827e-03 & --& 7.183858e-02 & --\\ 
0.129713 & 0 & 4322& 2.732537e-02 & 1.20& 3.083595e-02 & 1.07& 3.271674e-03 & 1.57& 3.088796e-03 & 1.46& 3.427624e-02 & 1.19\\ 
0.0657364 & 0 & 16642& 1.135351e-02 & 1.29& 1.347297e-02 & 1.22& 1.012115e-03 & 1.73& 9.847495e-04 & 1.68& 1.426008e-02 & 1.29\\ 
0.0329799 & 0 & 65282& 4.499879e-03 & 1.34& 5.454410e-03 & 1.31& 2.820949e-04 & 1.85& 2.784537e-04 & 1.83& 5.489740e-03 & 1.38\\ 
\midrule
0.241412 & 1 & 2202& 1.590592e-02 & --& 4.250048e-03 & --& 4.983831e-04 & --& 2.051919e-04 & --& 3.726391e-03 & --\\ 
0.129713 & 1 & 8202& 4.682540e-03 & 1.97& 1.069133e-03 & 2.22& 7.943574e-05 & 2.96& 2.543882e-05 & 3.36& 7.854022e-04 & 2.51\\ 
0.0657364 & 1 & 31602& 1.239505e-03 & 1.96& 2.457061e-04 & 2.16& 1.069243e-05 & 2.95& 2.458947e-06 & 3.44& 1.487627e-04 & 2.45\\ 
0.0329799 & 1 & 124002& 3.171237e-04 & 1.98& 5.656809e-05 & 2.13& 1.373283e-06 & 2.98& 2.383934e-07 & 3.38& 2.744157e-05 & 2.45\\ 
\midrule
0.241412 & 2 & 3242& 2.485494e-03 & --& 4.271783e-04 & --& 4.614337e-05 & --& 9.897321e-06 & --& 2.615419e-04 & --\\ 
0.129713 & 2 & 12082& 4.180162e-04 & 2.87& 7.454427e-05 & 2.81& 4.305935e-06 & 3.82& 9.444987e-07 & 3.78& 4.017522e-05 & 3.02\\ 
0.0657364 & 2 & 46562& 5.918798e-05 & 2.88& 1.066472e-05 & 2.86& 3.168190e-07 & 3.84& 7.066557e-08 & 3.81& 5.338156e-06 & 2.97\\ 
0.0329799 & 2 & 182722& 7.815297e-06 & 2.94& 1.411474e-06 & 2.93& 2.122297e-08 & 3.92& 4.763964e-09 & 3.91& 6.820794e-07 & 2.98\\ 
\bottomrule
\end{tabular}
  \end{adjustbox}
  \caption{Numerical results for the polytopal method of Section~\ref{sec:applications:polytopal}.\label{tab:polygonal}}
\end{table}


\section*{Acknowledgements}

Funded by the European Union (ERC Synergy, NEMESIS, project number 101115663).
Views and opinions expressed are however those of the authors only and do not necessarily reflect those of the European Union or the European Research Council Executive Agency. Neither the European Union nor the granting authority can be held responsible for them.


\printbibliography

\end{document}